\documentclass[11pt]{article}
\usepackage{amsmath}
\usepackage{amssymb}
\usepackage{amsfonts}
\usepackage{mathrsfs}
\usepackage{amsthm}
\usepackage{epsfig}

\usepackage{amsmath, amsthm, amssymb}
\usepackage[english]{babel}
\usepackage{graphics}
\usepackage{graphicx}
\usepackage{subfigure}
\usepackage{eucal}
\usepackage{amscd}
\usepackage{array}
\usepackage{verbatim}

\usepackage{multirow}

\newtheorem{thm}{Theorem}[section]

\newtheorem{lemma}[thm]{Lemma}
\newtheorem{prop}[thm]{Proposition}

\newtheorem{cor}[thm]{Corollary}
{\theoremstyle{definition}
}

{\theoremstyle{remark}
\newtheorem{rem}[thm]{Remark}}

{\theoremstyle{remark}
}

\newenvironment{proofof}[2]{\begin{proof}[Proof of #1 \ref{#2}.]}{\end{proof}}
\newcommand{\be}{\begin{equation}}
\newcommand{\ee}{\end{equation}}
\newcommand{\bes}{\begin{equation*}}
\newcommand{\ees}{\end{equation*}}

\newcommand{\ud}{\mathrm{d}}

\newcommand{\trajcdot}{}
\newcommand{\octcdot}{}

\newcommand{\Die}{D}

\newcommand{\st}{\, : \, }

\newcommand{\DD}{ {T_1 \mathbb{D}}}
\newcommand{\R}{\mathbb{R}}
\newcommand{\CC}{\mathbb{C}}

\newcommand{\Aff}{Af\!f}
\newcommand{\RP}{\mathbb{R}\mathbb{P}^1}
\newcommand{\CP}{\mathbb{C}\mathbb{P}^1}
\newcommand{\Disk}{\mathbb{D}}

\title{Geodesic flow on the Teichm\"uller disk of the regular octagon, cutting sequences and octagon continued fractions maps.}

\author{John Smillie\footnote{Partially supported by NSF Grant DMS-0901521.}
  \and  Corinna
Ulcigrai\footnote{Partially supported by an RCUK Academic Fellowship.} }
\date{}

\begin{document}

\maketitle

\begin{abstract}
In this paper we give a geometric interpretation of the renormalization algorithm and of the continued fraction map that we introduced in \cite{SU:sym} to give a characterization of symbolic sequences for linear flows in the regular octagon. We interpret this algorithm as renormalization on  the Teichm\"uller disk of the octagon and explain the relation with Teichm\"uller geodesic flow. This connection is analogous to the classical relation between Sturmian sequences, continued fractions and geodesic flow on the modular surface. We use this connection to construct the natural extension and the invariant measure for the continued fraction map. We also define an acceleration of the continued fraction map which has a finite invariant measure.
\end{abstract}

\section{Introduction}\label{introsec}
It is a classical fact that \emph{continued fractions} are related to the \emph{geodesic flow} on the modular surface (see for example \cite{Se:geo, Se:mod}). 
Moreover, there is a deep connection between continued fractions and \emph{Sturmian sequences}, i.e.~sequences which code bi-infinite linear trajectories in a square (we refer the reader to \cite{Ar:stur} for survey on Sturmian sequences). The connection between the direction of a trajectory on the square and the symbol sequence is made by means of continued fractions.  

In \cite{SU:sym} we considered the analogous problem  of characterizing symbolic sequences which arise in coding bi-infinite linear trajectories on the regular octagon and more generally on any regular $2n$-gon. We developed an appropriate version of the continued fraction algorithm, similar in spirit but different in detail from that introduced by Arnoux and Hubert (\cite{AH:fra}).
We used this continued fraction algorithm to relate symbol sequences and directions of trajectories.  Some of the main results from \cite{SU:sym} are summarized in the following subsections. 

In this paper, we make a connection between \emph{this continued fraction algorithm}, our analysis of \emph{symbolic sequences for the octagon} and the \emph{Teichm\"uller flow} on an appropriate Teichm\"uller curve (which will actually be an orbifold in our case) that plays the role of the modular surface. All three can be interpreted in the framework of renormalization   (the idea of renormalization is discussed in the  section \S\ref{renormschemessec}). This completes the parallel between Sturmian sequences, continued fractions and geodesic flow on the one hand and the corresponding objects for the regular octagon.

In the classical case there are two related maps that lead to continued fraction expansions, the Farey map and the Gauss map.
In  \cite{SU:sym} we construct the analogue of the Farey map, that we call \emph{octagon Farey map}. 
We use the connection with the geodesic flow to construct the natural extension and the invariant measure for this map. 
 The invariant measure for the octagon Farey map is an infinite measure. We also define an acceleration of the  octagon Farey map which we call the \emph{octagon Gauss map} and which, like the classical Gauss map, has a finite invariant measure.
The dynamics of our continued fraction algorithm is closely connected to the coding of  geodesic flows introduced by Caroline Series (see  \cite{Se:geoMc, Se:sym}) and we explain similarities and differences.

\paragraph{Outline.}
In the remainder of this section  we give some basic definitions and 
 we give a brief exposition of the renormalization schemes  and of the continued fraction algorithm introduced in \cite{SU:sym}. In \S\ref{discsec} we define the Teichm\"uller disk of a translation surface as a space of affine deformations. Our approach differs from the standard one since we consider also orientation reversing affine deformations. In \S\ref{TeichdiskrenormCF},  we first describe  the Veech group and an associated tessellation of the Teichm\"uller disk  of  the octagon (\S\ref{octagondisk}). We then give an interpretation of the renormalization schemes and of the continued fraction algorithm in  \S\ref{renormcutseq} in terms of a renormalization on the Teichm\"uller disk and explain the connection with the  Teichm\"uller geodesic flow. 
In the section \S\ref{natextsec}, we use this interpretation to find a natural extension and the absolutely continuous  invariant measure for our continued fraction map. We also explain the connection between the natural extension and a certain cross section of the  Teichm\"uller geodesic flow on the Teichm\"uller orbifold, in section \S\ref{crosssec}. Finally, in the section \S\ref{gausssec}, we define an acceleration of the continued fraction map which has a finite invariant measure.

\subsection{Basic definitions} \label{basicsec}
\subsubsection{Translation surfaces and linear trajectories}\label{transsurfsec}
A translation surface is a collection of polygons $P_j\subset\R^2$ with identifications of pairs of parallel sides so that (1) sides are identified by maps which are restrictions of translations, (2) every side is identified to some other side and (3) when two sides are identified the outward pointing normals point in opposite directions. If $\sim$ denotes the equivalence relation coming from identification of sides then we define surface $S=\bigcup P_j/\sim$.   
If $p$ is a point corresponding to vertexes of polygons,  the cone angle at $p$ is the sum of the angles at the corresponding points in the polygons $P_j$. We say that $p$ is \emph{singular} if the cone angle is greater than $2\pi$ and we denote by $\Sigma \subset S$ the set of singular points. 
For an alternative approach to translation surfaces see \cite{Ma:erg}. 

Our prime example of a translation surface is the following. Let ${O}$ be a regular octagon. The boundary of $O$ consists of four pairs of parallel sides. Let $S_O$ be the surface  obtained by identifying points of opposite parallel sides of the octagon $O$  by using the   isometry between them which is the restriction of a translation. The surface $S_O$   is an example of a translation surface which has genus 2 and a single singular point with a cone angle of $6\pi$. 

\begin{figure}
\centering
{
\includegraphics[width=0.3\textwidth]{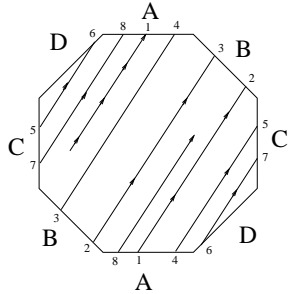}}
\caption{A linear trajectory in the octagon.\label{octagonflow}}
\end{figure}

If $S$ is a translation surface and $p\in S$ then the tangent space $T_p(S)$ has a natural identification with $\R^2$.  By using this identification any translation invariant geometric structure on $\R^2$ can be transported to all of $S\backslash \Sigma$. For example a vector $v\in\R^2$ gives a \emph{parallel vector field} on $S$,  a linear functional on $\R^2$ gives a \emph{parallel one form}. The metric $ds^2=dx^2+dy^2$ on $\R^2$ gives a \emph{flat metric} on $S$. 

Due to the presence of singular points parallel vector fields do not define flows. Nevertheless we can speak of trajectories for these vector field which we call \emph{linear trajectories}.  We call a  trajectory which does not hit $\Sigma$ \emph{bi-infinite}.  In this paper we consider only \emph{bi-infinite} trajectories, i.e.~trajectories which do not hit vertexes of the octagon.
A segment of a \emph{linear trajectory} in $S_O$ traveling in direction $\theta$ is shown in Figure \ref{octagonflow}: a point on the trajectory moves with  constant velocity vector making an angle $\theta$ with the horizontal and when it hits the boundary it re-enters the octagon at the corresponding point on the opposite side and continues traveling with the same velocity.

\subsubsection{Cutting sequences, admissibility and derivation}\label{cuttseqsec}

Given a linear trajectory in $S_O$ we would like to understand the sequence of sides of the octagon that it hits.
To do this let us label  each pair of opposite  sides of the octagon $O$ with a letter of an alphabet $\mathscr{A}=\{A,B,C,D\}$ as in Figure \ref{octagonflow}.
The \emph{cutting sequence} $c(\tau)$ associated to the linear trajectory $\tau$ is the bi-infinite word  in the letters of the alphabet $\mathscr{A}$, which is obtained by reading off the labels of the pairs of identified sides crossed by the trajectory $\tau$ as $t$ increases. 
 In \cite{SU:sym} we gave a characterization of bi-infinite words in the alphabet $\mathscr{A}$ which arise as cutting sequences of bi-infinite trajectories. Our characterization is based on the notion of admissible sequence and of derived sequence that we now recall.

 Let $\mathscr{A}^{\mathbb{Z}}$ be the space of bi-infinite words $w$ in the letters of $\mathscr{A}=\{A,B,C,D\}$. 
We call  \emph{transitions} the ordered pairs of letters which can occur in cutting sequences for trajectories with directions in a specified sector. Consider the diagrams $\mathscr{D}_0, \dots, \mathscr{D}_7$ in Figure \ref{diagrams}.
\begin{figure}
\centering
\subfigure[$\mathscr{D}_0$\label{D0}]{\includegraphics[width=0.2\textwidth]{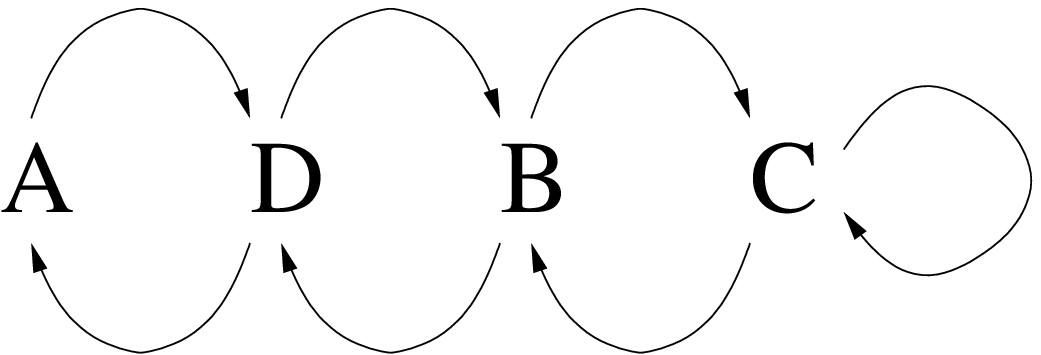}}
\hspace{3mm}
\subfigure[$\mathscr{D}_1$]{\includegraphics[width=0.2\textwidth]{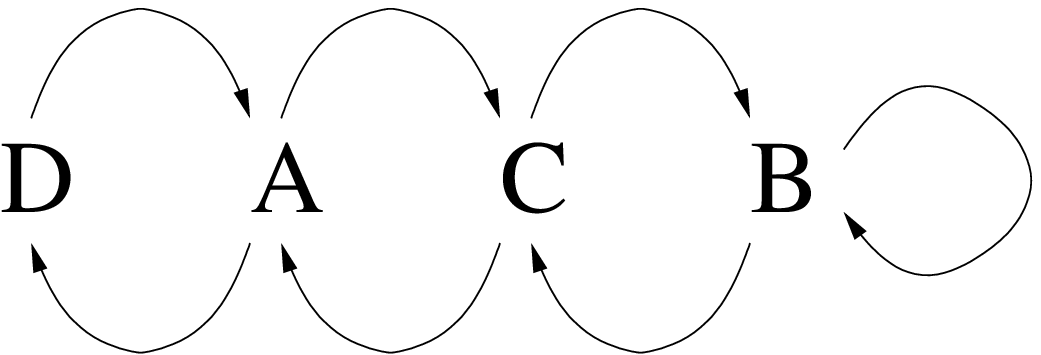}}
\hspace{3mm}
\subfigure[$\mathscr{D}_2$]{\includegraphics[width=0.2\textwidth]{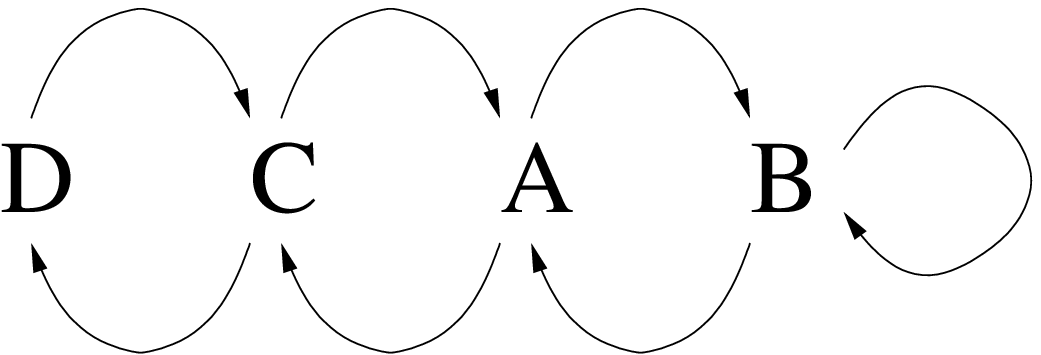}}
\hspace{3mm}
\subfigure[$\mathscr{D}_3$]{\includegraphics[width=0.2\textwidth]{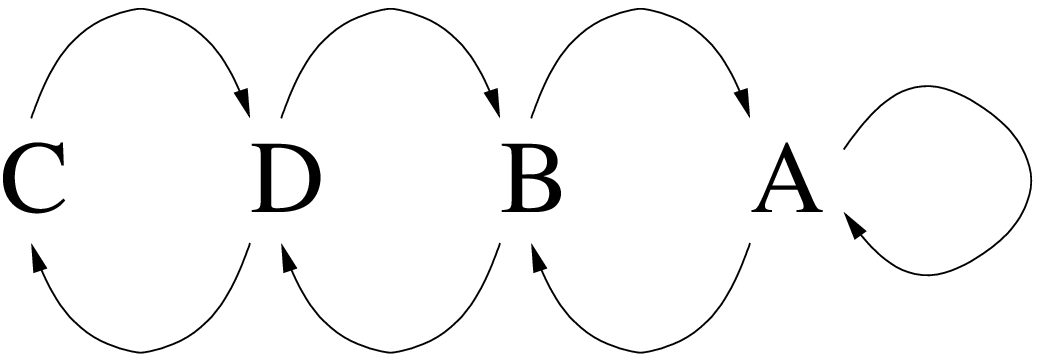}}
\hspace{3mm}
\subfigure[$\mathscr{D}_4$]{\includegraphics[width=0.2\textwidth]{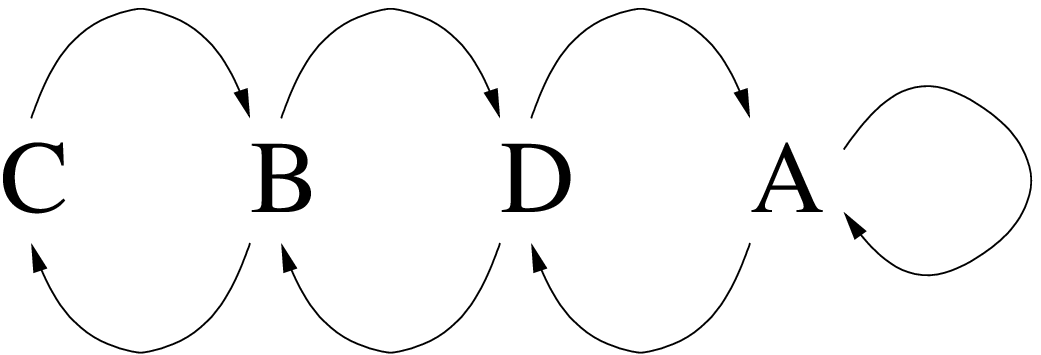}}
\hspace{3mm}
\subfigure[$\mathscr{D}_5$]{\includegraphics[width=0.2\textwidth]{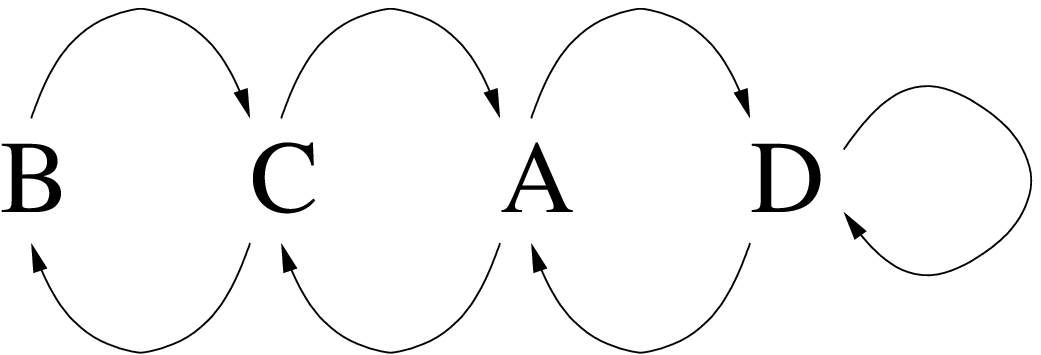}}
\hspace{3mm}
\subfigure[$\mathscr{D}_6$]{\includegraphics[width=0.2\textwidth]{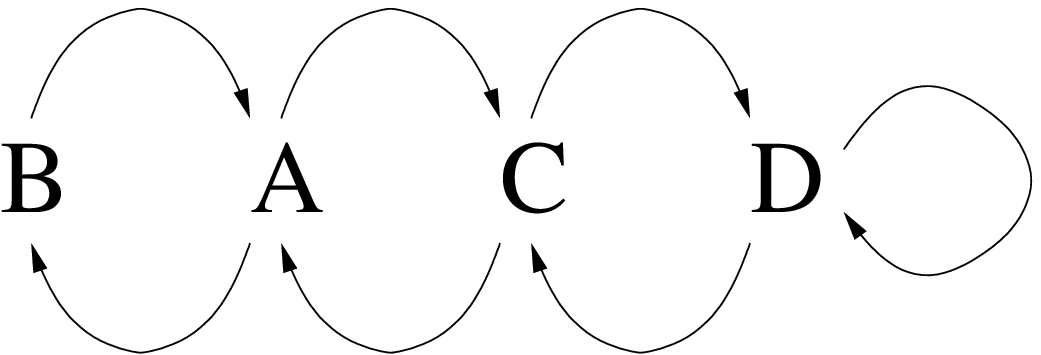}}
\hspace{3mm}
\subfigure[$\mathscr{D}_7$]{\includegraphics[width=0.2\textwidth]{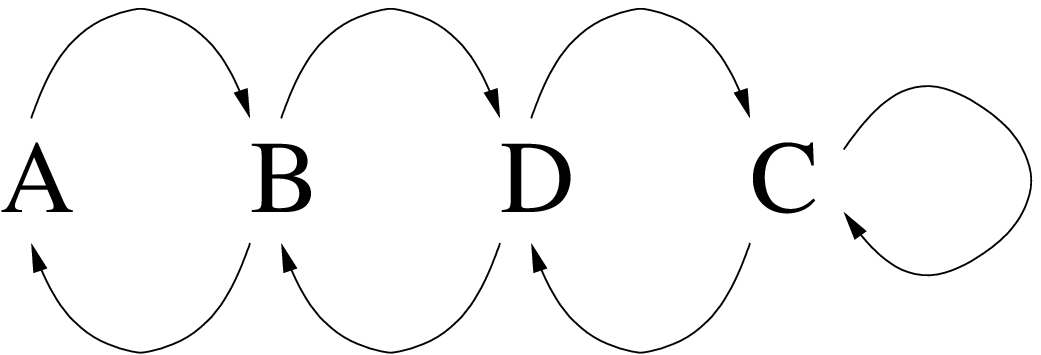}}
\caption{The transition diagrams $\mathscr{D}_0, \dots,\mathscr{D}_7$ corresponding to sectors  $\overline{\Sigma}_0,\dots, \overline{\Sigma}_7$. \label{diagrams}}
\end{figure}
 Let us say that the word  $w$ is \emph{admissible} if there exists a diagram $\mathscr{D}_i$ for $i\in \{0, \dots, 7\}$ such that all transitions in $w$ correspond to edges of $\mathscr{D}_i$. In this case, we will say that it is \emph{admissible in diagram $i$}. Equivalently, the sequence is admissible in diagram $i$ if it describes an infinite path on $\mathscr{D}_i$. We remark that some words are admissible in more than one diagram. 
It is not hard to check that  cutting sequences are admissible (see Lemma 2.3 in \cite{SU:sym}).   

We call \emph{derivation} the following combinatorial operation on admissible sequences. 
We say that a letter 
 in a  sequence  $w\in \mathscr{A}^{\mathbb{Z}}$ is \emph{sandwiched} if it is preceded and followed by the same letter. 
Given an admissible bi-infinite sequence $w \in \mathscr{A}^{\mathbb{Z}}$ the \emph{derived sequence}, which we denote by  $w'$, is the sequence obtained by keeping only the letters of $w$ which are sandwiched. 
For example, if $w$ contains the finite word $CACCCDBDCDC$ the derived sequence $w'$ contains the word $ACBCD$, since these are sandwiched letters in the string. Using the assumption that $w$ is admissible, one can show that $w'$ is again a bi-infinite sequence. 
A word $w \in \mathscr{A}^{\mathbb{Z}}$ is \emph{derivable} if it is admissible 
and its derived sequence $w'$ is admissible.  A word  $w \in \mathscr{A}^{\mathbb{Z}}$ is  \emph{infinitely derivable} if it is derivable and for each $n$ the result of deriving the sequence $n$ times is again admissible. 

In \cite{SU:sym} we proved that, given a cutting sequence of a trajectory $\tau$, the derived  sequence is again a cutting sequence (see Proposition 2.1.19 in \cite{SU:sym}) by defining a \emph{derived trajectory} $\tau'$ such that $c(\tau)'=c(\tau')$. The trajectory $\tau'$ is the result of applying a particular affine automorphism to the trajectory $\tau$. (Affine automorphisms are defined in section \ref{discsec}. The particular affine automorphisms used are  defined in section \ref{octagondisk}.) 
The following necessary condition on cutting sequences is a consequence. 
\begin{thm}\label{cutseqinfderiv}
A  cutting sequence $c(\tau)$ is infinitely derivable.
\end{thm}
\noindent Theorem \ref{cutseqinfderiv} gives a necessary condition for a sequence to be a cutting sequence. In \cite{SU:sym} we  also gave a sufficient condition for a sequence to be in the closure of the set of cutting sequences with respect to the natural topology on the sequence space. 

\subsection{Renormalization schemes.}\label{renormschemessec}

In this section we recall the definitions of three combinatorial algorithms introduced in \cite{SU:sym}, one acting on directions (the octagon Farey map in \S\ref{Fareymapsec}), one acting symbolically on cutting sequences of trajectories (recalled in \S\ref{combrenormsec})
 and the latter acting on trajectories (see \S\ref{trajrenormsec}). 
Each of these algorithms can be viewed as an example of the concept of renormalization. 

The idea of \emph{renormalization} is a key tool which has been used to investigate flows on translation surfaces.  This idea was borrowed from physics and it appears in one dimensional dynamics as well. The idea is to study the behavior of a dynamical system by including it into a space of dynamical systems and introducing a renormalization operator on this space of dynamical systems. This renormalization operator acts by rescaling both time and space. The long term behavior of our original system can be analyzed in terms of the behavior of the corresponding point under the action of the renormalization operator. Let us call the choice of a space of dynamical systems and of a renormalization operator a ``renormalization scheme". Examples of renormalization schemes include  the  Teichm\"uller flow on the moduli space of translation surfaces and Rauzy-Veech induction on the space of interval exchange transformations. The study of linear flows on the torus can be put in this framework:
we can think of the geodesic flow on the modular surface and the continued fraction algorithm as renormalization schemes. The geodesic flow on the modular surface is a special case of the Teichm\"uller flow and the continued fraction algorithm is closely related to the Rauzy-Veech induction.
In this paper we define a renormalization scheme suited to study linear trajectories on the octagon (see  \S\ref{TeichdiskrenormCF}), where the renormalization operator  acts by \emph{affine automorphisms} which gives a discrete approximation of the Teichm\"uller geodesic flow on the appropriate Teichm\"uller orbifold of the octagon. 
In \S\ref{TeichdiskrenormCF} we 
 we make the connection between this renormalization scheme and the three algorithms described below.

\subsubsection{The Octagon Farey map}\label{Fareymapsec}
We now define the \emph{octagon Farey map}. Let $\RP$ be the space of lines in $\R^2$.  There are two coordinate systems on $\RP$ which will prove to be useful in what follows.  The first is the inverse slope coordinate, $u$. A line in $\R^2$ is determined by a non-zero column vector with coordinates $x$ and $y$. We set $u((x,y))=x/y$. 
A linear transformation of $\R^2$ induces a projective transformation of $\RP$. The group of projective transformations of $\RP$ is $PGL(2,\R)$ and the kernel of the natural homomorphism from $GL(2,\R)$ to $PGL(2,\R)$ consists of $\pm Id$. The elements of $PGL(2,\R)$ correspond to linear fractional transformations.
 If $L= \left(\begin{smallmatrix}a& b \\ c& d  \end{smallmatrix} \right)$ is a matrix in $GL(2, \mathbb{R})$, we will denote by $L[x]$ the associated linear fractional  transformation, given by $ L[x]= \frac{ax+b}{cx+d}$. This linear fractional transformation records the action of $L$ on the space of directions in inverse slope coordinates.
 The second useful coordinate is the angle coordinate $\theta \in [0, \pi]$.  Where $\theta$ corresponds to the line generated by the vector with coordinates $x=\cos(\theta)$ and $y=\sin(\theta)$. Note that since we are parametrizing lines rather than vectors $\theta$ runs from $0$ to $\pi$ rather than from $0$ to $2\pi$. 

An interval in $\RP$ corresponds to a collection of lines in $\R^2$. Following the conventions of \cite{SU:sym} we will think of such an interval as corresponding to a sector in the upper half plane.
We will denote by $\Sigma_i$ the  sector of $\RP$ corresponding to the angle coordinate sectors $ \left[ i \pi /8 ,  (i+1)\pi/8 \right)$ for $i=0,\dots, 7$, each  of length $\pi/8$ in $[0,\pi)$. Let us stress that $\Sigma_i$  is a sector in  $\RP$ and we will abuse notation by writing $u \in \Sigma_i$ or $\theta \in \Sigma_i$, meaning that the coordinates belong to the corresponding interval of coordinates.

The isometry group of the octagon is the dihedral group $\Die_8$ (of order 16). This group acts on $\RP$ but the action is not faithful. The center of $\Die_8$ consists in the usual matrix representation of $\Die_8$ of $\pm I$, where $I$ denotes the identity matrix, and it acts trivially. The group $\Die_8/\{ \pm I\}$ acts faithfully and this group is isomorphic to $\Die_4$.
  Let $\overline{\Sigma}_i : = [i \pi /8 ,  (i+1)\pi/8]$ the corresponding closed intervals. The set ${\overline\Sigma}_0$ is a  fundamental domain for the action of the dihedral group $\Die_4$.
Let $\nu_i\in\Die_8$ be the isometry which sends $\overline{\Sigma}_i$ to $\overline{\Sigma}_0$.
These elements are represented by the following matrices:
\be \label{nujdef}
\begin{split} &
\nu_0= \begin{pmatrix} 1 & 0 \\ 0 & 1 \end{pmatrix} \phantom{-} \, \, 
\nu_1= \begin{pmatrix}  \frac{1}{\sqrt{2}} &  \frac{1}{\sqrt{2}} \\  \frac{1}{\sqrt{2}} & - \frac{1}{\sqrt{2}} \end{pmatrix}\, \, 
\nu_2= \begin{pmatrix} \frac{1}{\sqrt{2}} &  \frac{1}{\sqrt{2}} \\ - \frac{1}{\sqrt{2}} &  \frac{1}{\sqrt{2}} \end{pmatrix}  \, \, 
\phantom{-} 
\nu_3=\begin{pmatrix}0& 1 \\ 1& 0  \end{pmatrix} \\ & 
\nu_4=\begin{pmatrix} 0 & 1 \\ -1 & 0 \end{pmatrix} \, \, 
\nu_5= \begin{pmatrix} - \frac{1}{\sqrt{2}} &  \frac{1}{\sqrt{2}} \\  \frac{1}{\sqrt{2}} &  \frac{1}{\sqrt{2}} 
\end{pmatrix}\, \, 
\nu_6=  \begin{pmatrix} - \frac{1}{\sqrt{2}} &  \frac{1}{\sqrt{2}} \\ - \frac{1}{\sqrt{2}} & - \frac{1}{\sqrt{2}} \end{pmatrix}\, \, 
\nu_7= \begin{pmatrix} -1 & 0 \\ 0 & 1 \end{pmatrix}.
\end{split}
\ee
\noindent  The images of these matrices in $PGL(2,\R)$ give the group $\Die_4$. When we refer to group operations using $\nu_j$ we are thinking of their image in $PGL(2,\R)$.

Let
\bes
\gamma : = \begin{pmatrix} -1 &  2(1+\sqrt{2})  \\ 0  & 1  \end{pmatrix}. 
\ees
Let us denote by $\Sigma:= \Sigma_1 \cup \dots \cup \Sigma_7$. The map $\gamma$ takes sector $\overline{\Sigma}_0$ to the union $\overline{\Sigma}$.

 Let $F_i:\overline{\Sigma}_i\to\RP$ be the map induced by the linear map $\gamma\nu_i$ . Define the \emph{octagon Farey map} $F : \RP\to\RP$ to be the piecewise-projective map, whose action on the sector of directions corresponding to $\overline{\Sigma}_i$ is given by $F_i$. Observe that the maps $F_i$ fit together so that the resulting map $F$ is continuous.
 In the inverse slope coordinates  $F$ is a piecewise linear fractional transformation. If $ u  \in  \bar \Sigma_i$, we define
\bes \label{Fareymapdefinvslopes}
 F(u)  = \gamma \nu_i [u ] =  \frac{a_i u + b_i}{c_i u + d_i}, \qquad \mathrm{where} \, \, \begin{pmatrix} a_i & b_i \\ c_i & d_i \end{pmatrix}:= \gamma \nu_i  . 
\ees
The action in angle coordinates is obtained by conjugating  by conjugating by $\cot : [0,\pi] \rightarrow \mathbb{R}$, so that if $\theta \in \Sigma_i$ we have $F (\theta) =   \cot^{-1}\left( \frac{a_1 \cot(\theta) + b_i}{c_i \cot(\theta) + d_i}\right)$.
The graph of $F$ in angle coordinates is shown in Figure \ref{Fareygraph}. The map $F$ is pointwise expanding, but not uniformly expanding since the expansion constant tends to one at the endpoints of the sectors. Since each branch $F_i$ of $F$ is monotonic,  the inverse maps $F_i^{-1}:\Sigma\to\Sigma_i$, $ i=0, \dots, 7$ are well defined. 
\begin{figure}
\centering
{\includegraphics[width=0.5\textwidth,
]{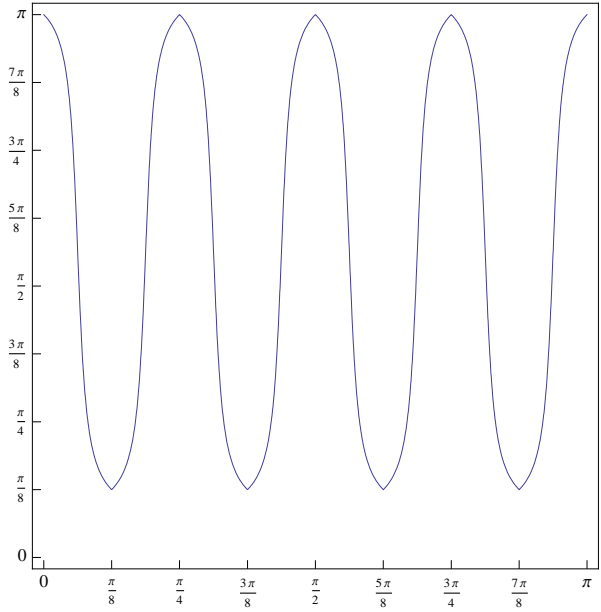}}
\caption{The graph of the octagon Farey map $F$.\label{Fareygraph}}
\end{figure}

The map $F$ gives an additive continued fraction algorithm for numbers in the interval $[0,\pi]$ as follows. 
 Let $S^*$ to be the set of all sequences $s=\{s_k\}$  for $k\ge 0$ that satisfy the condition $s_k\in \{0,\dots, 7\} $ and $s_k=0$ implies $k=0$. Given  $\{s_k\} \in S^* $, one can check that intersection $ \bigcap_{k \in \mathbb{N}} F_{s_0 }^{-1} F_{s_1 }^{-1} \cdots  F_{s_k}^{-1} [0, \pi] . $ is non empty and consists of a single point $\theta$. In this case we write
\be\label{sectorCFdef}
\theta = [s_0; s_1, \dots, s_k, \dots ]_O := 
 \bigcap_{k \in \mathbb{N}} F_{s_0 }^{-1} F_{s_1 }^{-1} \cdots  F_{s_k}^{-1} [0, \pi]  
\ee
 and say that $[s_0; s_1, \ldots ]$ is an \emph{octagon additive continued fraction expansion} of $\theta$ and that $s_0,s_1,\ldots$ are the \emph{entries} of the expansion of $\theta$. 
Let us call a direction $\theta$  \emph{terminating} if the  continued fraction expansion entries  of $\theta$ are eventually $1$ or eventually $7$.   With the exception of $\pi/8$ and $\pi$,   terminating  directions have two octagon additive continued fraction expansions. 

\subsubsection{Combinatorial renormalization and direction recognition.}\label{combrenormsec}
In this section we  describe the combinatorial renormalization scheme on cutting sequences of trajectories introduced in \cite{SU:sym}. This scheme is based on the operation of derivation on sequences but a key feature of this renormalization is that after deriving a sequence we put the sequence into a \emph{normal form} as explained below. The use of this convention  allows us to use the  continued fraction map in \S \ref{Fareymapsec} and to relate derivation to the geodesic flow on the Teichm\"uller orbifold of the octagon (see \S\ref{crosssec}). 

A permutation $\pi $ of $\mathscr{A}=\{A, B, C, D\}$ acts on sequences $w \in \mathscr{A}^{\mathbb{Z}}$ by permuting the letters of $w$ according to $\pi$; we denote this action by $\pi \cdot w$. 
We consider the following eight permutations of $\{A,B,C,D\} $ 
\bes
\begin{array}{llll}
\pi_0= Id & \pi_1 = (AD)(BC) & \pi_2 = (ABCD) & \pi_3= (AC)\\
\pi_4=(AC)(BD) & \pi_5 = (AB)(CD) & \pi_6 = (ADCB) & \pi_7 = (BD),
\end{array}
\ees
which are induced (see \cite{SU:sym})  by the action of the isometries $\nu_1, \dots, \nu_7 $ on cutting sequences, in the following  sense. 
Let $\nu_k\trajcdot \tau$ be the trajectory obtained by postcomposing $\tau$ with the isometry $\nu_k$. Then:
\be \label{relabelingeq}
c(\nu_k \trajcdot \tau) = \pi_k \cdot c(\tau), \qquad  \forall \quad 0\leq k \leq 7 .
\ee
\noindent Note that the central element of $\Die_8$ induces the trivial permutation so the image of $\Die_8$ in the permutation group is isomorphic to $\Die_4$.

Say that a word $w$ is admissible in a unique diagram $\mathscr{D}_j$. The \emph{normal form} of $w$ is the word $n(w):=\pi_j \cdot w_0$. We now define inductively a sequence of words $w_j$.
Set $w_0:= w$. Let us assume for now that $w_0$ is admissible in an \emph{unique} diagram with index $d_0$. The first step of the renormalization scheme  consist of taking the derived sequence of the normal form of $w_0$  and setting  $w_1 := n(w_0)'$. 
If after $k$ iterations the sequence obtained is again admissible in an \emph{unique} diagram, we  define
\be\label{renormalizedseq}
 w_{k+1}:= n(w_k)'.
\ee
In \cite{SU:sym} we showed that if $w$ is a \emph{non-periodic cutting sequence}, then the renormalization scheme described above is well defined for all $k\in \mathbb{N}$.   In this case, $w$ determines a unique infinite sequence $\{d_k\}_{k\in \mathbb{N}}$ such that $w_k$ is admissible in diagram  $d_k$. 
 This  sequence of admissible diagrams of a cutting sequence $c(\tau)$ can be used to recover the direction of the trajectory $\tau$ through the octagon additive continued fraction expansion, as follows.
\begin{thm}[Direction recognition, \cite{SU:sym}]\label{directionsthm}
If $w$ is a {non-periodic} cutting sequence, the direction of trajectories $\tau$ such that $w=c(\tau)$
is uniquely determined and given by 
$\theta=[d_0;d_1,\ldots, d_k, \ldots  ]$,  
where  $\{d_k\}_{k\in \mathbb{N}}$ 
is the sequence of admissible diagrams. 
\end{thm}
 \noindent The proof of the Theorem appears in \cite{SU:sym}, combining Proposition 2.2.1 and Theorem 2.3.1.

\subsubsection{A renormalization scheme for trajectories.}\label{trajrenormsec}
Given a trajectory  $\tau$ in direction $\theta$, let $s(\tau)=k \in \{ 0, \dots, 8 \}$ be such that $\theta \in \Sigma_{k}$. We say that $s(\tau)$ is the \emph{sector} of $\tau$.  If the \emph{sector} of $\tau$ is $k$, the cutting sequence  $c(\tau)$ is \emph{admissible} in diagram $k$. 
Given a trajectory $\tau$ with  $s(\tau)= k$, we define the \emph{normal form} of the trajectory $\tau$ the trajectory $n(\tau):= \nu_k\trajcdot \tau$, obtained by postcomposing $\tau$ with $\nu_k$, which is the isometry that maps $\overline{\Sigma}_k$ to $\overline{\Sigma}_0$. 
The renormalization scheme on the space of trajectories is obtained by alternately putting in normal form and deriving trajectories as follows. 
Given a trajectory $\tau$, let us recursively define its sequence of renormalized trajectories $\{\tau_k\}_{k\in \mathbb{N}}$ by: 
\be\label{renormalizationtrajectories}
\tau_0:= \tau; \qquad \tau_{k+1}:= n( \tau_k)', \quad k\in \mathbb{N}. 
\ee
An explicit definition of the derived trajectory $n( \tau_k)'$ is given in (\ref{derivedtraj}) and uses the affine automorphism $\Psi_\gamma$ described in \S\ref{octagondisk}. 
Given a trajectory $\tau$, we denote by $\{s_k(\tau)\}_{k\in \mathbb{N}}$ the \emph{sequence of sectors} given by
$ s_k(\tau):= s(\tau_k)$, where $s(\tau_k)$ is the sector of the $k^{th}$ renormalized trajectory. In words, the sequence $\{s_k(\tau)\}_{k \in \mathbb{N}}$ is obtained  by 
recording the sectors of the renormalized trajectories. 
This renormalization operation on trajectories is a natural counterpart to the combinatorial renormalization and of the continued fraction, as shown by the following Proposition (see \cite{SU:sym} for the proof).
\begin{prop}\label{relationsrenormalizations}
Let $w$ be the cutting sequence $c(\tau)$ of a trajectory $\tau$ in direction $\theta$. Assume that the $k^{th}$ renormalized sequence $w_k$ given by (\ref{renormalizedseq}) is well defined. Then
$c(\tau_k) = w_k$.

Moreover, for each $k\in\mathbb{N}$, the direction $\theta_k$ of the $k^{th}$ renormalized trajectory $\tau_k$ is given by $\theta_k = F^k (\theta)$
where $F$ is the octagon Farey map. In particular, the sequence $\{s_k(\tau)\}_{k\in\mathbb{N}}$ of sectors of a trajectory $\tau$  in direction $\theta$ coincides with the itinerary of $\theta$ under the octagon Farey map $F$,
 i.e.~$F^{k} (\theta) \in \Sigma_{s_k(\tau)}$ for all  $k\in \mathbb{N}$. 
\end{prop}

\section{Moduli spaces, the Teichm\"uller disk and the Veech group}\label{discsec}

The term moduli space is often used to refer specifically to the space of conformal structures on surfaces. We will use it here in the more general sense of a topological space which parametrizes a family of geometric structures on a surface. We will be interested in two moduli spaces in particular. Given a translation surface $S$ we want to consider the space of affinely equivalent (marked) translation surfaces up to translation equivalence and  the space of (marked) translation surfaces up to isometry (see section  \S\ref{Teichdisksec}). The second space is the Teichm\"uller disk of $S$ and the first is closely related to the $SL(2,\R)$ orbit of the surface in a stratum. 
  Our approach differs from the standard one in a number of respects. In particular we consider both orientation preserving an orientation reversing affine automorphisms. We are lead to do this because orientation reversing affine automorphisms play and important role in our renormalization schemes.  A second novelty is that we construct our moduli space directly and not as a subset of a larger ``stratum" of translation surfaces. In \S\ref{Teichorbsec} we define the Teichm\"uller orbifold and in \S\ref{isoDelaunaysec} we define the iso-Delaunay tessellation of the Teichm\"uller disk.

\paragraph{Affine diffeomorphisms.} 
Let $S$ and $S'$ be translation surfaces and let $\Sigma$ and $\Sigma'$ the respective conical singularities sets.  Consider a homeomorphism $\Phi$ from $S$ to $S'$ which takes $\Sigma$ to $\Sigma'$ and is a diffeomorphism outside of $\Sigma$. We can identify the derivative $D\Phi_p$ with an element of $GL(2,\mathbb{R})$. We say that $\Phi$ is an \emph{affine diffeomorphism} if the $D\Phi_p$ does not depend on the point $p$. In this case we write $D\Phi$ for $D\Phi_p$. Clearly an affine diffeomorphism $\Phi : S \rightarrow S'$ sends infinite linear trajectories on $S$ to infinite linear trajectories on $ S'$. 
If $\tau$ is a linear trajectory on $S$,   we  denote by $\Phi {\tau}$ the linear trajectory on $ S'$ which is obtained by composing $\tau$ with $\Phi$.

We say that $S$ and $S'$ are \emph{affinely equivalent} if there is an affine diffeomorphism $\Phi$ between them.  
We say that $S$ and $S'$ are \emph{isometric} if they are affinely equivalent with $D\Phi\in O(2)$. We say that $S$ and $S'$ are \emph{translation equivalent} if they are affinely equivalent with $D\Phi=Id$.  If $S=\bigcup P_j/\sim$ and $S'=\bigcup_k P'_k/\sim$  then a translation equivalence $\Upsilon$ from $S$ to $S'$ can be given by a \emph{cutting and pasting} map, that is to say, we can subdivide the polygons $P_j$ into smaller polygons and define a map $\Upsilon$ so that  the restriction of $\Upsilon$ to each of these smaller polygons is a translation and the image of $\Upsilon$ is the collection of polygons $P'_k$. We require that appropriate identifications be respected.
A concrete example is given in \S\ref{octagondisk}, see Figure \ref{shearedoctagon1}. 

An affine diffeomorphism from $S$ to itself is an \emph{affine automorphism}. The collection of affine diffeomorphisms is a group which we denote by $\Aff(S)$. 
The collection of isometries of $S$ is a finite subgroup of $\Aff(S)$ and  the collection of translation equivalences is a subgroup of the group of isometries. 

\paragraph{The canonical map.}
Let $S$ be a translation surface given by  $S =\bigcup P_j/\sim$.  Given  $\nu \in GL(2,\R)$, we denote by $\nu  P \subset \mathbb{R}^2$ the image of a polygon $P \subset \mathbb{R}^2$ under the linear map $\nu$.   The translation surface $S'=\nu \cdot S$ is obtained by gluing the corresponding sides of $\nu \octcdot P_1, \dots, \nu \octcdot P_n $. There is a \emph{canonical map} $\Phi_\nu$ from the surface $S$ to the surface $\nu \cdot S$ which is given by the restriction of the linear map $\nu$ to the polygons $P_1 , \dots ,  P_n$.
There is a connection between canonical maps and affine automorphisms.
We can concretely realize an affine automorphism of $S$ with derivative $\nu$ as a composition of the canonical map $\Phi_\nu: S\to\nu\cdot S$ with a translation equivalence, or cutting and pasting map, $\Upsilon:\nu\cdot S \to S$.

\paragraph{The Veech group.}
The \emph{Veech homomorphism} is the homomorphism $\Psi\mapsto D\Psi$  from $\Aff(S)$ to $GL(2,\R)$.  The kernel of the Veech homomorphism is  the finite group of translation equivalences of $S$.  
\begin{rem}\label{trivialkernel} The kernel of the Veech homomorphism is trivial if and only if an affine automorphism of $S$ is determined by its derivative. Moreover, if  the kernel of the Veech homomorphism is trivial, given $\Psi \in V(S)$ with derivative $\nu$ there is a \emph{unique} translation equivalence  $\Upsilon$ such that $\Psi = \Upsilon \Phi_{\nu}$ where $\Phi_{\nu}$ is the canonical map.
  \end{rem}
\noindent 
The image of the Veech homomorphism is  a discrete subgroup of the subgroup  $SL_{\pm}(2,\mathbb{R})$ of matrices with determinant $\pm1$. 
We call the image \emph{Veech group} and denote it by $V(S)$.  
We write $V^+(S)$ for the image of the group of \emph{orientation preserving} affine automorphisms in $SL(2,\R)$.  We denote the image of $V(S)$ (respectively  $V^+(S)$) in $PGL(2,\R)$ by $V_P(S)$ (respectively $V_P^+(S)$).  We note that the term Veech group is used by most authors to refer to the group that we call $V^+(S)$. Some authors use the term Veech group to refer to the  the group $V^+_P(S)$. Since we will make essential use of orientation reversing affine automorphisms we use the term Veech group for the larger group $V(S)$.   

A translation surface $S$ is called a \emph{lattice surface} if  $V^+(S)$ is a lattice in $SL(2, \mathbb{R})$ or, equivalently, $V(S)$ is a lattice in $SL_{\pm}(2, \mathbb{R})$. The torus $T^2=\mathbb{R}^2 / \mathbb{Z}^2$ is an example of a lattice surface whose  Veech group is $GL(2, \mathbb{Z})$. Veech proved  more generally that all  translation surfaces obtained from regular polygons are lattice surfaces, see \cite{Ve:tei}.

\paragraph{Delaunay triangulations.}\label{Delaunaysec}
Let ${T}$ be a triangulation of a translation surface $S$.  (When we speak of triangulations we do not demand that distinct edges have distinct vertexes. We allow decompositions into triangles so that the lift to the universal cover is an actual triangulation. These decompositions are called $\Delta$-complexes in \cite{Ha:alg}.) A natural class of triangulations to consider are those for which the edges are saddle connections (see \cite{MS:Hau}). A typical such triangulation will have long and thin triangles. The following class of triangulations have triangles which tend to have diameters which are not too large.

A triangulation ${T}$ is a \emph{Delaunay triangulation} of $S$ if (1) each side of a triangle is a geodesic w.~r.~t.~the flat metric of $S$, (2) the vertexes of the triangulation are contained in the singularity set $\Sigma$ and (3) for each triangle, there is an immersed euclidean disk which contains on its boundary the three vertexes of the triangle and does not contain in its interior any other singular point of $\Sigma$. The last condition is the \emph{Delaunay condition}. 
Another equivalent way of expressing condition (3), used e.~g.~by Rivin \cite{Ri:euc} and Bowman \cite{Bo:vee}, is the following $(3)'$ for each for each side $e$ of $T$ and pair of triangles $t_1$ and $t_2$ which share  $e$ as a common edge, the \emph{dihedral angle} $d(e)$, which is the sum of the two angles opposite to $e$ in $t_1$ and $t_2$ respectively, satisfies $d(e) \leq \pi$. We refer to \S\ref{octagondisk} for concrete examples of Delaunay triangulations. 

Let us call a \emph{Delaunay switch} the move from a triangulation $T$ to a new triangulation $T'$ which is obtained by replacing an edge $e$ shared by two triangles $t_1$ and $t_2$ with the opposite diagonal of the quadrilateral formed by $t_1$ and $t_2$. Starting from any $T$ satisfying (1) and (2) one can obtain a Delaunay triangulation by a finite series of Delaunay switches which decrease the dihedral angles. 
Thus, each flat surface admits a Delaunay triangulation, not necessarily unique (for example, Figure \ref{triangles1}, \ref{triangles2} give two Delaunay triangulations of $S_O$).  
The Delaunay triangulations fails to be unique exactly when there is an immersed disk which contains four points or more points of $\Sigma$ on its boundary. When this happens, some dihedral angle $d(e)$ is equal to $\pi$, since a quadrilateral inscribed in a circle is part of the triangulation. 

\subsection{The Teichm\"uller disk of a translation surface}\label{Teichdisksec}
In this section we will describe  the Teichm\"uller disk of a translation surface $S$ as a space of marked translation surfaces.

Let $S$ be a translation surface.  Consider a triples $f:S\to S'$ where $f$ is an affine diffeomorphism and the area of $S'$ is equal to the area of $S$. We say that the translation surface $S'$ is \emph{marked} (by $S$).  Using the convention that a map determines its range and domain we can identify a triple with a map and denote it by $[f]$. We say two triples  $f:S\to S'$ and $g:S\to S''$ are equivalent if there is a translation equivalence $h:S'\to S''$ such that $g=fh$. 
Let ${\cal{\tilde M}}_A(S)$ be the set of equivalence classes of triples. We call this the set of \emph{marked translation surfaces affinely equivalent to $S$}. There is a canonical basepoint corresponding to the identity map $id:S\to S$. 

\begin{prop} \label{affineequiv} The set ${\cal{\tilde M}}_A(S)$ can be canonically identified with $SL_\pm(2,\mathbb{R})$.
\end{prop}
\begin{proof} Given $\eta\in SL_{\pm}(2,\mathbb{R})$ and a translation surface $S$, we described  at the beginning of section \S\ref{discsec} a new translation surface $\eta\cdot S$ and a canonical map $\Phi_\eta:S\to \eta\cdot S$ where $D\Phi_\eta=\eta$. Define a function $\iota_1$ from $SL_\pm(2,\mathbb{R})$ to the set of triples which takes $\eta$ to the triple $\Phi_\eta: S\to \eta\cdot S$. Since the map $\eta$ multiplies area by a factor of $|\det  \eta|$ it follows that $area(S)=area(S')$.
We now define a map $\iota_2$ from triples to $SL_\pm(2,\mathbb{R})$ and show that $\iota_1\iota_2=\iota_2\iota_1=Id$. Let $f:S\to S'$ be a triple and let $\iota_2([f])$ be the matrix $Df$. Let us check that $\iota_2$ is well defined. If $[f]$ and $[g]$ are equivalent triples then there is an $h$ with $hf=g$ and $Dh=Id$. Applying the chain rule we have $DhDf=Dg$ so $Df=Dg$.
The composition $\iota_2 \iota_1$ takes $\eta$ to $D\Phi_\eta$. Since $D\Phi_\eta=\eta$ this composition is the identity. The composition $\iota_1 \iota_2$ takes a triple $f:S\to S'$ to the triple $\Phi_\eta:S\to \eta\cdot S$ with $\eta=Df$. To show that this composition is the identity we need to show that these two triples are equivalent. Let $h=\Phi_\eta  f^{-1}$. By definition we have $hf=\Phi_\eta$. We need to check that $h$ is a translation equivalence. We have $Dh=D\Phi_\eta Df^{-1}=Df Df^{-1}=Id$.
\end{proof}

\paragraph{The $SL_\pm(2,\R)$ action and the Veech group action.}
There is  a natural \emph{left action} of the subgroup $SL_\pm(2,\R)\subset GL(2,\R) $ on ${\cal{\tilde M}}_A(S)$. Given a triple $f:S\to S'$ and an $\eta\in SL_\pm(2,\R)$,   we consider the canonical map $\Phi_\eta:S'\to S''$ defined at the beginning of \S\ref{discsec}. We get the action by sending $[f]$ to $\Phi_\eta  f:S\to S''$. 
The previous proposition shows that this action acts simply transitively on ${\cal{\tilde M}}_A(S)$. Using the identification of ${\cal{\tilde M}}_A(S)$ with $SL_\pm(2,\R)$ this action corresponds to left multiplication by $\eta$.
There is a natural \emph{right action} of $\Aff(S)$ on the set of triples. Given an affine automorphism $\Psi:S\to S$ we send $f:S\to S'$ to $f \Psi:S\to S'$. This action induces a right action of $V(S)$ on ${\cal{\tilde M}}_A(S)$. Using the identification of ${\cal{\tilde M}}_A(S)$ with $SL_\pm(2,\R)$ this action corresponds to right multiplication by $D\Psi$. It follows from the associativity of composition of functions that \emph{these two actions commute}.

\paragraph{Isometry classes}
We would also like to consider marked translation surfaces up to \emph{isometry}. We say that two triples $f:S\to S'$ and $g:S\to S''$ are equivalent up to isometry if there is an isometry $h:S'\to S''$ such that $g=fh$.  Let ${\cal{\tilde M}}_{I}(S)$ be the collection of isometry classes of triples. 
Let us denote by  $\mathbb{H}$ the upper half plane, i.e.~$\{ z \in \mathbb{C} |\,  \Im z >0\}$ and by $\mathbb{D}$ the unit disk, i.e.~$\{ z \in \mathbb{C} | \, |z|<  1 \}$. In what follows, we will identify  them by the conformal map $\phi: \mathbb{H} \rightarrow \mathbb{D}$ given by $\phi(z) = \frac{z-i}{z+i}$. 
\begin{prop} \label{isometryequiv} The space ${\cal{\tilde M}}_{I}(S)$ of marked translation surfaces up to isometry is isomorphic to $\mathbb{H}$ (hence to $\mathbb{D}$). 
\end{prop}
\begin{proof} We will define an explicit map from isometry classes of marked translation surfaces $[f]$ to $\mathbb{H}$. This gives an explicit map to $\mathbb{D}$ by postcomposing with the conformal map $\phi$ above. 
Let  $f:S\to S'$ represent an element of ${\cal{\tilde M}}_{I}(S)$.  The map $f$ induces a linear isomorphism from $T(S)$ to $T(S')$. All equivalent triples representing the same element of ${\cal{\tilde M}}_{I}(S)$ determine the same metric on $T(S')$. We can pull this metric back to get a metric on $T(S)=\R^2$. A metric together with an orientation determines a complex structure. A complex structure is induced by an $\R$ linear map from $\R^2$ to $\CC$. Two such maps give the same complex structure if they differ by post-composition by multiplication by a non-zero complex number. We can express  this linear map by a matrix $\begin{pmatrix}z_1&z_2\end{pmatrix}$. The matrix $\lambda\begin{pmatrix}z_1&z_2\end{pmatrix}=\begin{pmatrix}\lambda z_1&\lambda z_2\end{pmatrix}$ gives the same complex structure. Thus we can identify the space of complex structures with $\CP$. The space of complex structures corresponds to the subset of row vectors whose entries are linearly independent over $\R$. This is just the complement of the image of $\RP$ in $\CP$ under the natural inclusion.
In order to identify this set with a subset of $\CC$ we choose a chart for $\CP$. There are two standard charts to use based on the fact that $\begin{pmatrix}z_1&z_2\end{pmatrix}$ can be written as $\begin{pmatrix}1&z_2/z_1\end{pmatrix}$ or $\begin{pmatrix}z_1/z_2&1\end{pmatrix}$. Let  $\phi_j:\CC\to\CP$, $j=1,2$ be $\phi_1(z)= \begin{pmatrix} z&1\end{pmatrix}$ and $\phi_2(z)= \begin{pmatrix} 1&z\end{pmatrix}$. We will use $\phi_1$ though $\phi_2$ is often used. 
  Thus the number $z_1/z_2$ is an invariant that determines the complex structure. We can identify the space of complex structures with the subset of projective space $\CP$ consisting of pairs of vectors that are linearly independent over $\R$. (The pairs that are linearly dependent over $\R$ correspond to the real axis.) This invariant takes values in either the upper half-plane or the lower half-plane depending on the orientation induced by the complex structure. Note that this takes the standard complex structure $\begin{pmatrix}1&i\end{pmatrix}$ to the complex number $-i$.
Each metric corresponds to two complex structures, one the standard orientation and one with the opposite orientation.  If we compose the map $Df$ with complex conjugation then we get the second complex structure.  If we write $\chi= z_1/z_2$ then the pair $\{\chi,\bar\chi\}$ determines the metric. We make the convention that we extract from the pair that element that lies in the upper half-plane.
\end{proof}
\noindent 
If we identify a triple $f:S\rightarrow S'$ with a matrix $\left(\begin{smallmatrix} a&b\\c&d\end{smallmatrix}\right)$ by Proposition \ref{affineequiv}, then the corresponding complex structure by Proposition \ref{isometryequiv} is given by the row matrix  $\begin{pmatrix}a+ci&b+di\end{pmatrix}$.
The corresponding element of $\CC$ under the chart $\phi_1$ is  $\frac{ai+c}{bi+ d}$ and the corresponding element of $\mathbb{D}$ is $\frac{(a-d)i+c+ b}{(a+d)i + c-b }$. 
The \emph{right action} of the Veech group on triples described above projects to an action of the Veech group on the space of complex structures which corresponds to \emph{pulling back} a complex structure. (This will be different from the action corresponding to \emph{pushing forward} a complex structure, which is induced by the \emph{left action} of $GL(2,\R)$.) The action of the Veech group on the space of complex structures is the projective action of $GL(2,\R)$ on row vectors coming from multiplication on the right, that is to say
$\begin{pmatrix} z_1 & z_2\end{pmatrix}\mapsto \begin{pmatrix} z_1 & z_2 \end{pmatrix} \left(\begin{smallmatrix} a & b \\ c & d\end{smallmatrix}\right)$. 
When the matrix $\nu = \left( \begin{smallmatrix} a & b \\ c & d\end{smallmatrix}\right)$ has positive determinant it takes the upper and lower half-planes to themselves and the formula is 
$z \mapsto \frac{az+c}{bz+ d}$. 
When the matrix $\nu$ has negative determinant the formula is $z \mapsto \frac{a \overline{z}+c}{b \overline{z}+ d}$.
The Veech group acts via isometries with respect to the hyperbolic metric of constant curvature on $\mathbb{H}$. The action on the unit disk can be obtained by conjugating by the conformal map $\phi: \mathbb{H} \rightarrow \mathbb{D}$.

The hyperbolic plane has a natural \emph{boundary}, which corresponds to  $\partial\mathbb{H} = \{ z \in \mathbb{C} | \, \Im z = 0\}\cup\{\infty\}$ or $\partial \mathbb{D}= \{ z \in \mathbb{C} | \, |z|=1\}$. The boundaries can be naturally identified with space of projective parallel one-forms on $S$. Projective parallel one-forms give examples of projective transverse measures which were use by Thurston to construct his compactification of Teichm\"uller space. A parallel one form gives us a measure transverse to the singular foliation defined by the kernel of the one-form. If two parallel one-forms differ by multiplication by a non-zero real scalar then they give the same projective transverse measure.
A parallel one-form corresponds to a linear map from $\R^2$ to $\R$ which we identify with a row vector $\begin{pmatrix} x_1&x_2\end{pmatrix}$. We can identify the space of projective parallel one forms with the corresponding projective space $\RP$. The linear map represented by $\begin{pmatrix} x_1&x_2\end{pmatrix}$ is sent by standard chart $\phi_1$  to the point $x_1/x_2 \in \mathbb{R} = \partial \mathbb{H}$ and by $\phi \phi_1$ (where the action of $\phi: \mathbb{H}\rightarrow \mathbb{D}$ is extends to the boundaries) to the point $e^{i \theta_x} \in \partial \mathbb{D}$ where $\sin  \theta_x = -2x_1x_2 /(x_1^2+x_2^2)$ and $\cos \theta_x= (x_1^2- x_2^2) /(x_1^2+x_2^2)$. 

The \emph{Teichm\"uller flow} is given be the action  of the $1$-parameter subgroup $g_t$ of $SL(2,\mathbb{R})$ given by the diagonal matrices $$g_t : = \begin{pmatrix} e^{t/2} & 0 \\  0 & e^{-t/2}  \end{pmatrix}$$ on ${\cal{\tilde M}}_{A}(S)$.  This flow acts on translation surfaces by rescaling the time parameter of the vertical flow and rescaling the space parameter for a transversal to the vertical flow thus we can view it as a renormalization operator. 
If we project ${\cal{\tilde M}}_{A}(S)$ to ${\cal{\tilde M}}_{I}(S)$ by sending a triple to its isometry class and  using the identification  ${\cal{\tilde M}}_{I}(S)$ with $ \mathbb{H}$ given in Proposition \ref{affineequiv}, then the Teichm\"uller flow corresponds to the hyperbolic geodesic flow:
\begin{lemma} \label{geodesic} 
Orbits of the $g_t$-action on ${\cal{\tilde M}}_{A}(S)$ project to geodesics in $\mathbb{H}$ parametrized at unit speed. Given
$\nu=\left( \begin{smallmatrix} a&b\\c&d \end{smallmatrix}\right)$, the geodesics through the marked translation surface $[\Phi_\nu]$ converges to the boundary point corresponding to the row vector $\begin{pmatrix} a&b\end{pmatrix}$ in positive time and to the boundary point corresponding to $\begin{pmatrix} c&d\end{pmatrix}$ in backward time.
\end{lemma}
\noindent We call a $g_t$-orbit in ${\cal{\tilde M}}_{A}(S)$ (or, under the identifications, in $T_1 \mathbb{D}$) a \emph{Teichm\"uller geodesic}. If   $\Phi_\nu: S \rightarrow \nu \cdot S$ represents the equivalence class $[f]$, the  parametrized \emph{Teichm\"uller geodesic} through $[f]$ is given by $ \{ [\Phi_{g_t \nu }] \}_{t \in \mathbb{R}}$. 

We get a map 
 from the space of marked translation surfaces ${\cal{\tilde M}}_{A}(S)$  to $T_1 \mathbb{H}$ (hence to $T_1 \mathbb{D}$) as follows. Given a triple $[f] \in {\cal{\tilde M}}_{A}(S)$, let  $\nu_f$ be the matrix representing $[f]$ (given by Proposition \ref{affineequiv}) and let $p_f$ be the point in $\mathbb{H}$ representing the isometry class of $[f]$ according to Proposition \ref{isometryequiv}. We send $[f]$ to  $(p_f, v_f) \in T_1 \mathbb{H}$ where $v_f$  is the  the derivative at $p_f$ of the Teichm\"uller geodesic through $[f]$. One gets a map from ${\cal{\tilde M}}_{A}(S)$ to $ T_1 \mathbb{D}$ using the identification of $ T_1\mathbb{H}$ with $T_1 \mathbb{D}$ induced by $u: \mathbb{H}\rightarrow \mathbb{D}$ and its derivative. 

\begin{lemma}\label{4to1} The map from ${\cal{\tilde M}}_{A}(S)$ to  $\rightarrow T_1 \mathbb{H} $ (or to $ T_1  \mathbb{D}$) described above is a surjective $4$ to $1$ map from ${\cal{\tilde M}}_{I}(S)$ to  $T_1 \mathbb{H}$ (respectively $T_1 \mathbb{D}$). 
\end{lemma}

\begin{proof} 
Given a triple $f: S \rightarrow S'$, let $ e_1, e_2$ be the canonical base of $\mathbb{R}^2 \approx T_1 S'$, and let $v_1, v_2 \in \mathbb{R}^2\approx T_1 S$  be the pull-back under $f$ of $e_1, e_2$ respectively.   Let $[f]= [\Phi_\nu]$ where $\nu =\left( \begin{smallmatrix}
a&b\\ c&d \end{smallmatrix}\right)$ and let $\ell_1=\begin{pmatrix} a&b\end{pmatrix}$  and $\ell_2=\begin{pmatrix}c&d\end{pmatrix}$ be forward and backward endpoints of the geodesics through $[f]$ (recall Lemma \ref{geodesic}). 
If we had replace either $\ell_1$ or $\ell_2$ by its negatives then we would get (by the correspondence in Proposition \ref{affineequiv}) a marked translation surface  for which the corresponding point of $\mathbb{D}$ or $\mathbb{H}$ flows to the same points on the boundary under the forward/backward geodesic flow and hence another marked translation surface which is represented by the same tangent vector at that point. By switching the signs of $\ell_1$ and $\ell_2$ we get four such marked translation surfaces, which correspond to the original surface, the surface obtained by reversing the direction $v_2$ (which is the direction maximally contracted along the geodesic flow), the one obtained reversing the direction $v_1$ (which is normal to the direction of maximal contraction of the geodesic flow) and the surface obtained by rotating the plane $v_1, v_2$ by $180^\circ$.
\end{proof}

\subsection{The Teichm\"uller curve or orbifold of a translation surface.}\label{Teichorbsec}
The quotient of ${\cal{\tilde M}}_{I}(S)$ by the natural right action of the Veech group $V(S)$ is the moduli space of unmarked translation surfaces which we call ${\cal{M}}_{I}(S) = {\cal{\tilde M}}_{I}(S) / V(S)$. This space is usually called the Teichm\"uller curve associated to $S$. In our case, since we allow orientation reversing automorphisms this quotient might be a surface with boundary so the term Teichm\"uller curve does not seem appropriate. Instead we call it the Teichm\"uller orbifold associated to $S$ (see Thurston's notes \cite{Th:thr} for a discussion of orbifolds). We denote by ${\cal{ M}}_{A}(S)$ the quotient ${\cal{\tilde M}}_{A}(S)/ V(S)$ of ${\cal{\tilde M}}_{A}(S)$ by the right action of the Veech group. This space is a four-fold cover of the tangent bundle to ${\cal{M}}_{I}(S)$ in the sense of orbifolds (see Lemma \ref{4to1}). The Teichm\"uller flow on ${\cal{ M}}_{A}(S)$ can be identified with the geodesic flow on the Teichm\"uller orbifold. We note that in the particular case where the space ${\cal{M}}_{I}(S)$ is a polygon in the hyperbolic plane the geodesic flow in the sense of orbifolds is just the \emph{hyperbolic billiard flow} on the polygon which is to say that if we project an orbit of this flow to the polygon then it gives a path which is a hyperbolic geodesic path except where it hits the boundary and when it does hit the boundary it bounces so that the angle of incidence is equal to the angle of reflection.

\subsection{The Iso-Delaunay tessellation of the Teichm\"uller disk.}\label{isoDelaunaysec}

Let $f:S\to S'$ be a triple. A Delaunay triangulation of $S'$ can be pulled back by $f$ to give an affine triangulation of $S$.
If we fix an affine triangulation $T$ of $S$ we can consider the collection of triples $f:S\to S'$ for which $T$ is the pullback of a Delaunay triangulation of $S'$. This is a closed  subset (possibly empty) of the space of triples ${\cal{\tilde M}}_{A}(S)$. Since the property of being a Delaunay triangulation depends only on the isometry class of a surface (or in other words it is invariant under the left action of $O(2)$ on ${\cal{\tilde M}}_{A}(S)$) this set is a subset of  ${\cal{\tilde M}}_{I}(S)$ and hence, by the identification in Proposition \ref{isometryequiv}, of the hyperbolic plane $\mathbb{D}$ (or $\mathbb{H}$).
This subset of $\mathbb{D}$ is convex and bounded by a finite number of geodesic segments and is called \emph{iso-Delaunay tile}. The collection of all such sets gives a tiling of the Teichm\"uller disk $\mathbb{D}$ which is called the \emph{Iso-Delaunay tessellation} (following Veech \cite{Ve:del}, see the exposition by Bowman \cite{Bo:vee}). An example of an \emph{Iso-Delaunay tessellation} is shown in Figure \ref{tessellation}, see Proposition \ref{isoDelaunayoctagon}. When one crosses transversally the boundary of  two adjacent iso-Delaunay regions, the Delaunay triangulations change abruptly and one can show that they change by a certain number of Delaunay switches (see Figures \ref{triangulations}, \ref{triangulationsL} and \S \ref{octagondisk} for some concrete examples). While the interior of the tiles correspond to  triples  $f:S\to S'$ for which $S'$ admits a unique Delaunay triangulation, points on the boundary of the iso-Delaunay tessellation correspond exactly to triples  $f:S\to S'$ for which $S'$ admits more than one Delaunay triangulation. 

\section{Renormalization schemes on the Teichm\"uller disk.}\label{TeichdiskrenormCF}
\subsection{The Teichm\"uller disk of the octagon} \label{octagondisk}
Let $S_O$ be the translation surface obtained by identifying opposite sides of the octagon $O$.  Let us first describe the group $\Aff(S_O)$ as well as $V(S_O)$. We refer to \cite{SU:sym} for further details. 

\paragraph{The Veech group  and the affine automorphism group of $S_O$.}
Since we allow orientation reversing transformations, the entire isometry group $\Die_8$ of the octagon $O$ is contained in $\Aff(S_O)$. For $\eta\in D_8$ let $\Psi_\eta:S_O\to S_O$ denote the corresponding affine automorphism. We have $D\Psi_\eta=\eta$. Consider in particular  the reflection $\Psi_\alpha $ of the octagon at the horizontal axes and  the reflection $\Psi_\beta $ in the tilted line which forms an angle $\pi/8$ with the horizontal axis. These are given by the two matrices
\be \label{alphabetadef}
\alpha := \begin{pmatrix} 1 & 0 \\  0 & -1  \end{pmatrix} \qquad  \beta := \begin{pmatrix} \frac{\sqrt{2}}{2} & \frac{\sqrt{2}}{2} \\  \frac{\sqrt{2}}{2} & - \frac{\sqrt{2}}{2}  \end{pmatrix}   = \rho_{\frac{\pi}{8}}  \cdot \alpha \cdot \rho_{-\frac{\pi}{8}}
 \ee 
  where $\rho_\theta = \left(\begin{smallmatrix} \cos \theta & - \sin \theta \\  \sin \theta & \cos\theta \end{smallmatrix} \right)$  is the matrix representing counterclockwise rotation by the angle $\theta$. 

Another element of  $\Aff(S_O)$ was described by Veech. Consider the shear:
\be \label{sigmadef}
\sigma  = \begin{pmatrix} 1 &  2(1+ \sqrt{2}) \\ 0 &  1  \end{pmatrix}. \ee 
\begin{figure}
\centering
\includegraphics[width=.9\textwidth]{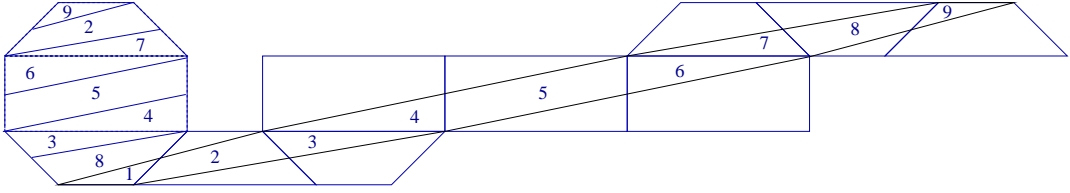}
\caption{The affine octagon $O'$ and the cut and paste map $\Upsilon_o$.\label{shearedoctagon1}}
\end{figure}

\noindent The image $O':= \sigma \octcdot O \subset {\mathbb{R}^2} $ is the affine octagon shown in Figure \ref{shearedoctagon1}; $O'$ can be mapped to the original octagon $O$ by cutting it into polygonal pieces, as indicated in Figure \ref{shearedoctagon1}, and rearranging these pieces without rotating them to form $O$ (in Figure \ref{shearedoctagon1} the pieces of $O$ and the pieces of $\sigma \octcdot O $ have been numbered to show this correspondence). Let us denote  by $\Upsilon_o : O' \rightarrow O$ this \emph{cut and paste} map. 
Clearly $D \Upsilon_o = Id$. The automorphism  $\Psi_{\sigma}$ is given by the composition $\Upsilon_o  \sigma$. 


In our renormalization scheme for trajectories (see \S \ref{trajrenormsec}) we use an orientation reversing element whose linear part is given by the following matrix:
\be \label{gammadef} 
\gamma :=  \begin{pmatrix} -1 &  2(1+\sqrt{2})  \\ 0  & 1
 \end{pmatrix} . 
\ee
Let us remark that $\gamma$ is an \emph{involution}, i.e. $\gamma^2=id$ or $\gamma=\gamma^{-1}$. One can  check that  $ \gamma \, \nu_7 = \sigma$, where $\nu_7$ is the reflection at the vertical axes, see (\ref{nujdef}). Thus,
 the action of $\gamma$ on $O$ is obtained by first reflecting it with respect to the vertical axis (this sends $O$ to $O$, but reverses the orientation), then shearing it through $\sigma$.  The image  $O':= \gamma \octcdot O$ is the same as in Figure \ref{shearedoctagon1}, but the orientation is reversed. 
  Since $\gamma$ is an involution, $\gamma \octcdot O'=O$. 
 If we compose  $\gamma : O \rightarrow O'$ with the cut and paste map  $\Upsilon_{o}: O' \rightarrow O$ (defined above) which rearranges  the pieces of the skewed octagon in Figure \ref{shearedoctagon1} to form the regular octagon,  we get an affine automorphism $\Psi_{\gamma}:= \Upsilon_{o} \gamma  $ of $S_O$.  
 An alternative description of $\Psi_{\gamma}$ is given in \cite{SU:sym}. The element $\Psi_{\gamma}$ plays a key role in the renormalization scheme for trajectories  in \S\ref{trajrenormsec}, since if the trajectory    $\tau$ has direction $\theta \in \Sigma_0$, the derived trajectory $\tau'$ such that $c(\tau)'=c(\tau')$ is given by
\be \label{derivedtraj}
\tau' = \Phi_\gamma \tau, \qquad  (\mathrm{when} \ s(\tau) = 0)
\ee
where $\Phi_\gamma \tau$ is a trajectory on $S_O$ obtained by post-composing the trajectory $\tau$ with  $\Phi_\gamma$. 
 
\begin{lemma} \label{Veechprop}The group $\Aff(S_O)$ is generated by the  affine diffeomorphisms $\Psi_\alpha$, $\Psi_\beta$ and $\Psi_\gamma$.  
 The Veech group $V(S_O )$ is generated by the corresponding linear maps  $\alpha$, $\beta$ and $\gamma$. The  kernel of the homomorphism from $\Aff(S_O)$ to $GL(2,\R)$ is trivial. 
\end{lemma}
\noindent The analogous Lemma for the group $V^+$ of double cover of the octagon  is due to Veech \cite{Ve:tei} and  for  $V^+(S_O)$ it is proved by Earle and Gardiner \cite{CG:tei}. The proof for $V(S_O )$ is included in \cite{SU:sym}.

\paragraph{The fundamental domain on the Teichm\"uller disk. }
Consider the identification of the set  ${\cal{\tilde M}}_A(S_O)$ of marked translation surfaces affinely equivalent to $S_O$  with $T_1\Disk$ described in section \S\ref{Teichdisksec}. 
Each unit tangent vector in $\DD$ corresponds to a matrix $\nu \in PSL(2,\mathbb{R})$ and to the marked affine deformation $\nu\cdot S_O$, given by the triple $\Phi_{\nu}: S_O \rightarrow \nu \cdot S_{ O}$. A point of $\Disk$ can be thought as an isometry class of a triple $f: S_O \rightarrow S'$ or equivalently as a \emph{metric} on $S_O$, the metric obtained pulling back the flat metric on $S'$ by $f$. Let  $\underline{0}$ be the center of the disk $\Disk$, which represents the canonical basepoint which is the isometry class of  the triple $id:S_O\to S_O$ and the standard flat metric on $S_O$.  

The Veech group $V_P(S_O)$ acts on $\Disk$ on the right as described in \S\ref{Teichdisksec}. Given any subset $\mathscr{D}\subset \Disk$, we will use the notation $\mathscr{D} \nu$ for the image of $\mathscr{D}$ under the right action of $\nu \in V_P(S_O)$. A fundamental domain for the action of $V_P(S_O)$ is given by the hyperbolic triangle $\mathscr{F}$ shown in light color in Figure \ref{tessellation}, with a vertex and an angle of $\pi/4$  at the center $\underline{0}$ of $\Disk$, one horizontal side and the other two vertexes on the boundary $ \partial \Disk$. 
The affine reflections $\alpha$, $\beta$ and $\gamma$ defined by (\ref{alphabetadef}) and (\ref{gammadef}) act (on the right) on $\Disk$ as  hyperbolic reflections through the sides of the triangle  $\mathscr{F}$, $\alpha$ being the reflection at the horizontal side and $\gamma$ being the reflection at the side  connecting the two points at infinity. We call this latter side $E_0$. 
Hence, $V_P(S_O)$ is the \emph{group of reflections at sides of the hyperbolic triangle $\mathscr{F}$}, or in other words the extended triangle group with signature $(4,\infty,\infty)$\footnote{Some authors use the term triangle group for groups of orientation preserving hyperbolic isometries and use the term extended triangle group for  the full group, see for example  \cite{SKa:Fuc}. This is a subgroup of index two in the extended triangle group.}. As we saw in \S\ref{discsec}, the quotient  $\Disk / V_P(S_O) $ is isomorphic to  the Teichm\"uller orbifold ${\cal{M}}_{I}(S_O)$ (see \S\ref{Teichorbsec}). 

The images of the fundamental domain $\mathscr{F}$ under $V_P(S_O)$ give  a  tiling of $\Disk$. The eight images of $\mathscr{F}$ by the action of the group elements $\nu_0, \nu_1,\dots, \nu_7$ (see (\ref{nujdef}) for their definition) have the  center $\underline{0}$ of $\Disk$ as a common vertex and form a hyperbolic octagon with vertexes on $\partial \Disk$ (see Figure \ref{tessellation}), which we denote by $\mathscr{O}$. 
Let us define\footnote{Note the difference with the relation $\Sigma_i = \nu_i^{-1}\Sigma_0$ that we obtain considering the linear action of $\nu_i$ on  $\mathbb{R}^2$.} 
\be\label{sidesrightaction}
 E_i : = E_0 \nu_i, 
\ee
so that $E_0, \dots, E_7$ the sides of $\mathscr{O}$, where $E_0$ is the side in common with $\mathscr{F}$, and the remaining sides are obtained by moving clockwise around the octagon. Each $E_i$ is a hyperbolic geodesic connecting two ideal vertexes. 
 The images of $\mathscr{O}$ under $V_P(S_O)$ 
  give a tessellation of $\Disk$, which we call \emph{ideal octagon tessellation}. The ideal octagon tessellation is obtained from the tessellation in Figure \ref{tessellation}  by collecting into a single tile all sets of hyperbolic triangles which share a common vertex in $\Disk$.  This is the tessellation shown in gray in Figure \ref{dualtree}.

\begin{figure}
\centering
\includegraphics[width=.6\textwidth]{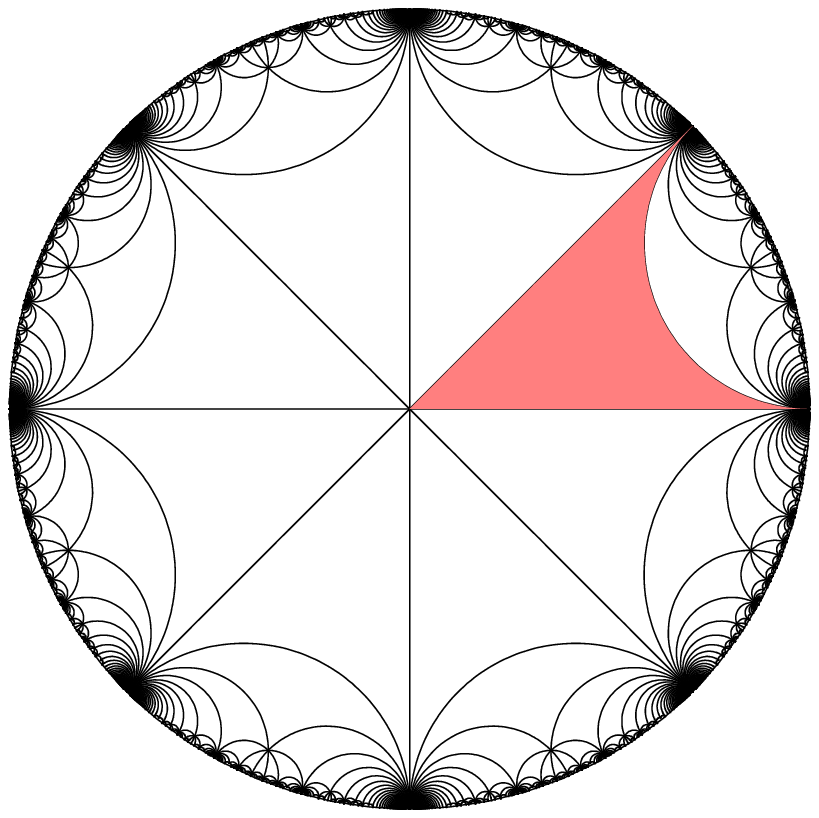}
\caption{The hyperbolic tessellation of $\Disk$ by copies of $\mathscr{F}$.\label{tessellation}}
\end{figure}  

\paragraph{The Iso-Delaunay tessellation of the Teichm\"uller disk of $S_O$.}
Let us describe the iso-Delaunay tessellation of the Teichm\"uller disk of $S_O$. 

 \begin{prop}\label{isoDelaunayoctagon}
The iso-Delaunay tessellation of the Teichm\"uller disk of $S_O$ is the tessellation by images of the triangle $\mathscr{F}$  by the right action of the elements of $V_P(S_O)$. 
\end{prop}
\noindent  For the proof of Proposition \ref{isoDelaunayoctagon}, we refer the reader to Veech \cite{Ve:bic} or Bowman \cite{Bo:vee}. We describe here  some examples of Delaunay triangulations corresponding to different tiles. 
Since any point in the Teichm\"uller disk corresponds to a triple $f: S_O \rightarrow S'$  (see \S\ref{Teichdisksec}) or equivalently to an affine deformation $S'$ marked by $S_O$, we will describe at the same time the triangulations on $O$ which correspond to the pull-back via the marking $f$ of the Delaunay triangulations on $S$. 

The surface $S_O$  is an example of a surface for which the Delaunay triangulation is not unique: all the triangulations of  the octagon $O \subset \mathbb{R}^2$ by Euclidean triangles are Delaunay triangulations (for example the triangulations in Figures \ref{triangles1} and \ref{triangles2}). Let $g_t^{\theta} $ be the subgroup conjugate to the Teichm\"uller geodesic flow which acts by contracting the direction $\theta$ when $t>0$ which is given by $g_t^{\theta} : = \rho_{\pi/2-\theta}^{-1} \, g_t \, \rho_{\pi/2 -\theta}$ (recall that $\rho_\theta$ is the counterclockwise rotation by $\theta$). If we consider a deformation  $ g_t ^{\theta} \cdot S_O$ where $t>0$ is small and  $\theta \neq k \pi/8, k\in \mathbb{N}$, there is a unique Delaunay triangulation.
The Delaunay triangulations and their pull-backs corresponding to triples in the interior of $\mathscr{F}$ and in $ \mathscr{F}\nu_1  $  (for which $S'= g_t ^{\theta} \cdot S_O$ is a small deformation with $\theta \in \Sigma_0$  and $\theta \in \Sigma_1$ respectively) are shown  in Figure \ref{triangles1}, \ref{triangles1b} and \ref{triangles2}, \ref{triangles2} respectively. Figures \ref{triangles1b}, \ref{triangles2b} show the actual Delaunay triangulations on $S'= g_t ^{\theta} \cdot S_O$, while  Figures \ref{triangles1}, \ref{triangles2} show their pull-backs to affine triangulations of $S_O$. Let us remark that the triangulations in Figure \ref{triangles1} and  \ref{triangles2} 
 differ by  simultaneous Delaunay switches involving the three the edges which are not sides of the octagon $O$. Similarly, pull-backs of triangulations associated with triples in $ \mathscr{F}\nu_i$ are images of the one in  Figure \ref{triangles1} by the standard linear  action of $\nu_i^{-1}$, $i=1, \dots, 7$ on $\mathbb{R}^2$. 
\begin{figure}[ht]
 \begin{center}
\subfigure[ $S$  for $\theta \in \Sigma_0$  ]{\includegraphics[width=0.15\textwidth]{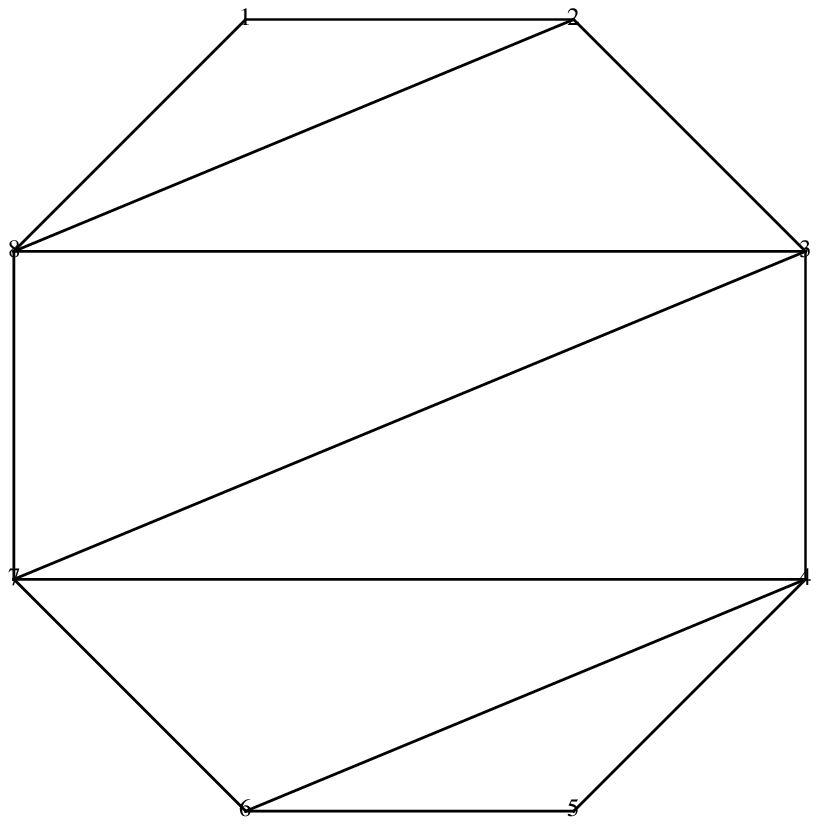}\label{triangles1}}
\hspace{2mm}
\subfigure[ $S'$  for $\theta \in \Sigma_0$ ]{\includegraphics[width=0.2\textwidth]{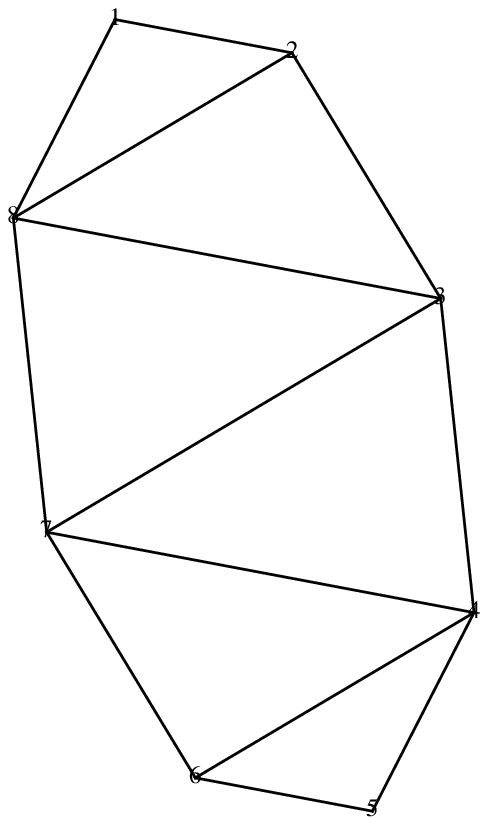}\label{triangles1b}}
\hspace{2cm}
\subfigure[ $S$  for $\theta \in \Sigma_1$ ]{\includegraphics[width=0.15\textwidth]{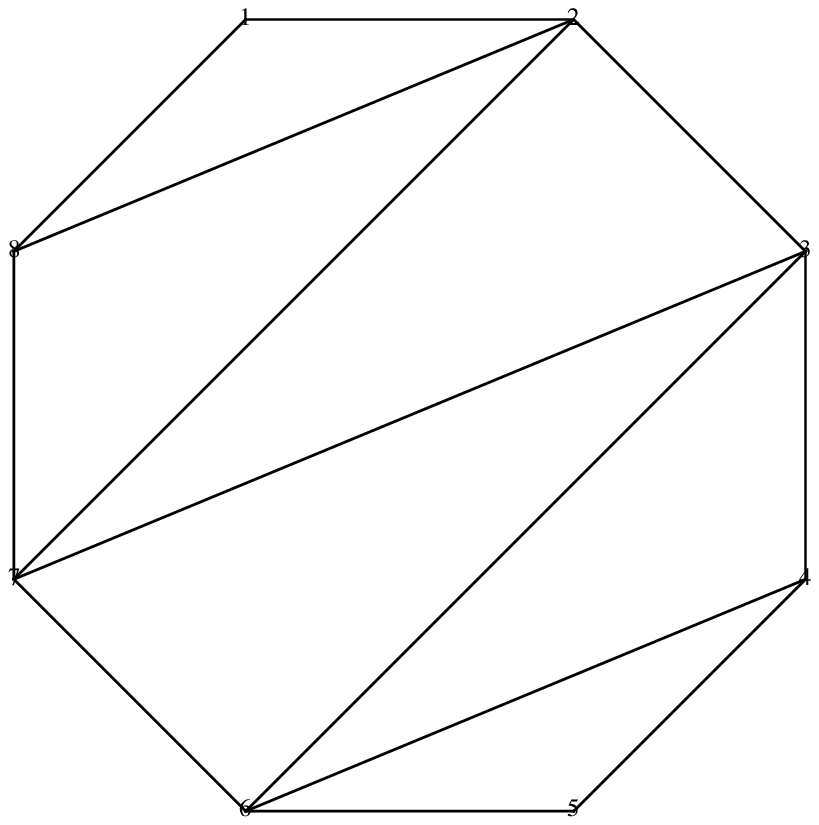}\label{triangles2}}
\hspace{3mm}
\subfigure[ $S'$  for $\theta \in \Sigma_1$]{\includegraphics[width=0.19\textwidth]{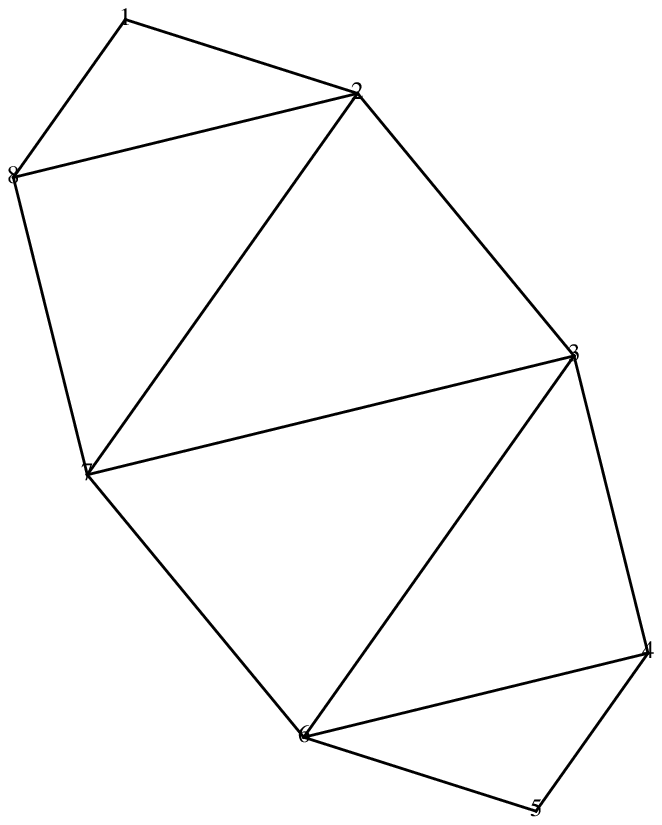}\label{triangles2b}}
\end{center}
\begin{center}
\subfigure[ $S$ for $\theta = \frac{\pi}{8}$ ]{\includegraphics[width=0.15\textwidth]{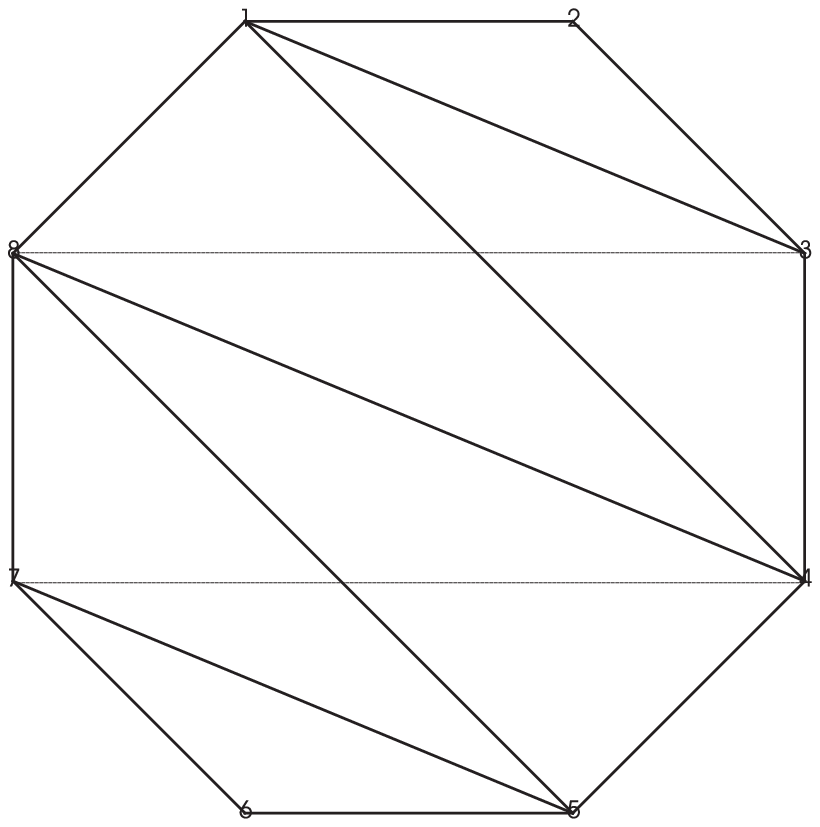}\label{triangles12}}
\hspace{9mm}
\subfigure[ $S'$ for $\theta =\frac{\pi}{8}$ ]{\includegraphics[width=0.14\textwidth]{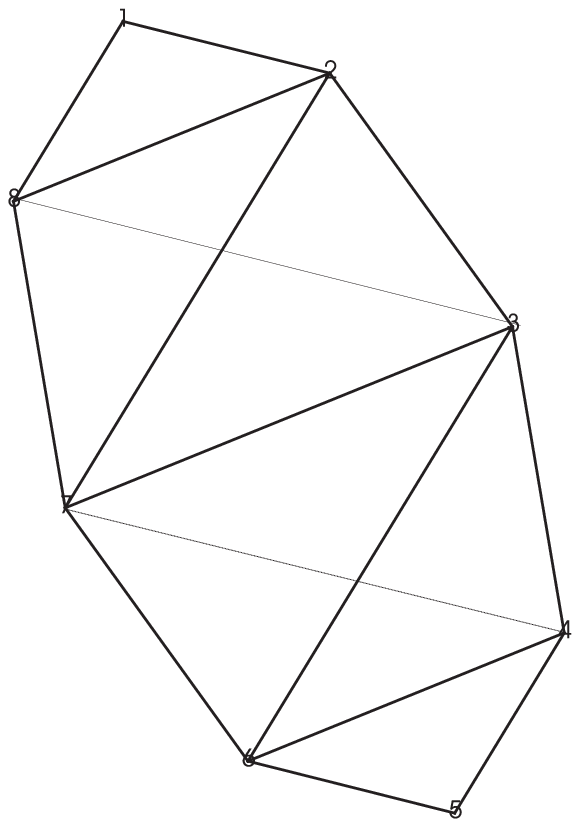}\label{triangles12b}}
\hspace{1mm}
\caption{Delaunay triangulations and their pull-backs when $S' = g_t^{\theta} \cdot S$ for a small $t>0$.\label{triangulations}}
\end{center}
\end{figure}
The deformations $S'= g_t ^{\theta} \cdot S_O$ in the direction $\theta = k \pi/8$ correspond to moving along one of geodesic rays through the center  of $\mathbb{D}$ which belong to the boundary of the iso-Delaunay tessellation. These eight rays are exactly the rays which limit to an ideal vertex of $\mathscr{O}$.  These deformations preserves exactly one of the axes of symmetry of $O$ (see Figure \ref{triangles12b}, where the axes has direction $\pi/8$).  Thus, the corresponding deformation $S'$ admits more than one Delaunay triangulations, for example any triangulation  obtain using the sides drawn in Figure \ref{triangles12b}. Figure \ref{triangles12} shows the affine pull-back of these sides and shows that they contain both the triangulations in Figures \ref{triangles1} and \ref{triangles2}, which are   symmetric with respect to the axes in direction $\pi/8$.  

Let us show that (marked) translation surfaces represented by  points on a side of $\mathscr{O}$ admit more than one  Delaunay triangulation.  Let $E$ be a side of $\mathscr{O}$ and let consider the two directions $\pi k/8$ and $\pi (k+1)/8$ which correspond to the rays limiting to the endpoints of $E$ on $\partial \Disk$. Let us first recall that triples $f: S_O \rightarrow S'$  correspond to metrics on  $S_O$, obtained as pull-back by $f$ of the flat metric on $S'$. The metrics corresponding to triples represented by points belonging to the side  $E \subset \Disk$  are precisely the metrics which make the two directions $\pi k/8$ and $\pi (k+1)/8$ perpendicular (equivalently, the directions obtained as image of the directions $\pi k/8$ and $\pi (k+1)/8$ under $f$ are orthogonal with respect to the flat metric on $S'$). 
\begin{figure}
 \begin{center}
\subfigure[$O$ pull-back ]{\includegraphics[height=0.225\textheight]{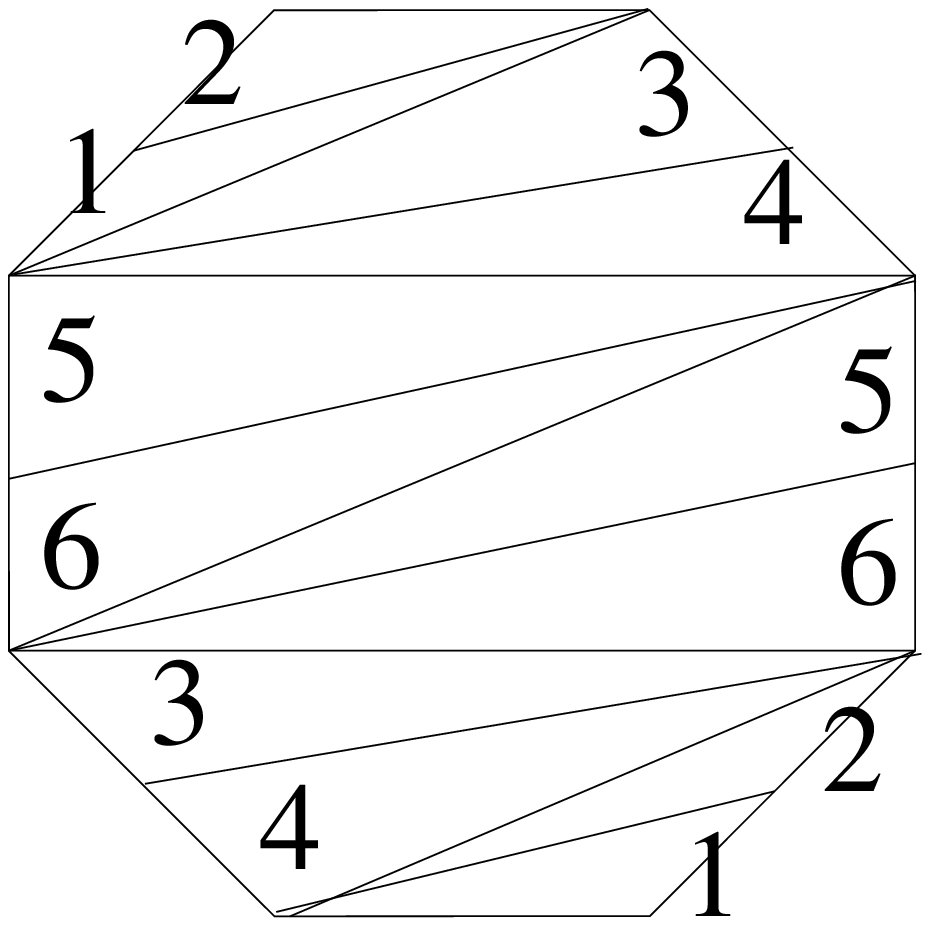}\label{71}}
\hspace{1mm}
\subfigure[$O$ decomposed as $L$]{\includegraphics[height=0.225\textheight]{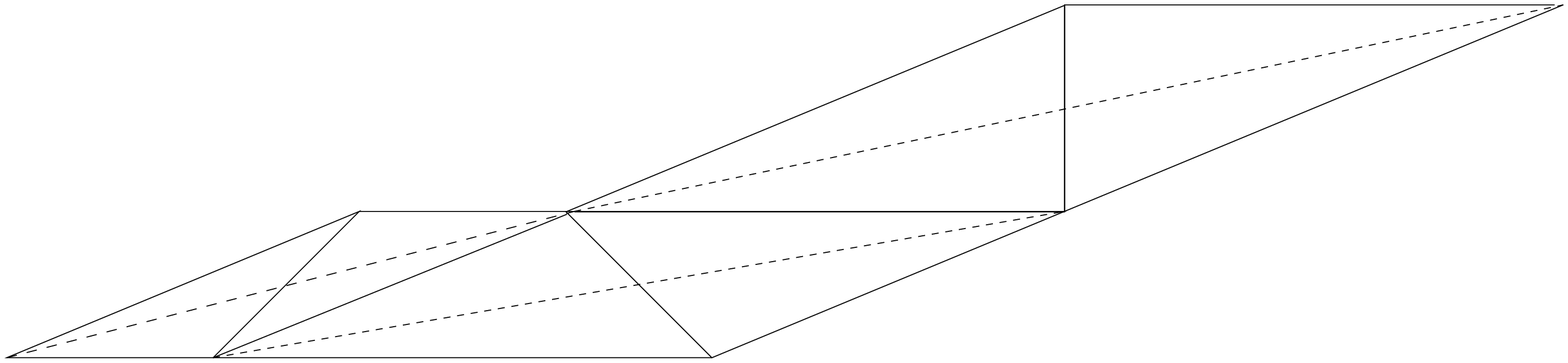}\label{72}}
\hspace{1mm}
\subfigure[$S'$ as $L$]{\includegraphics[height=0.225\textheight]{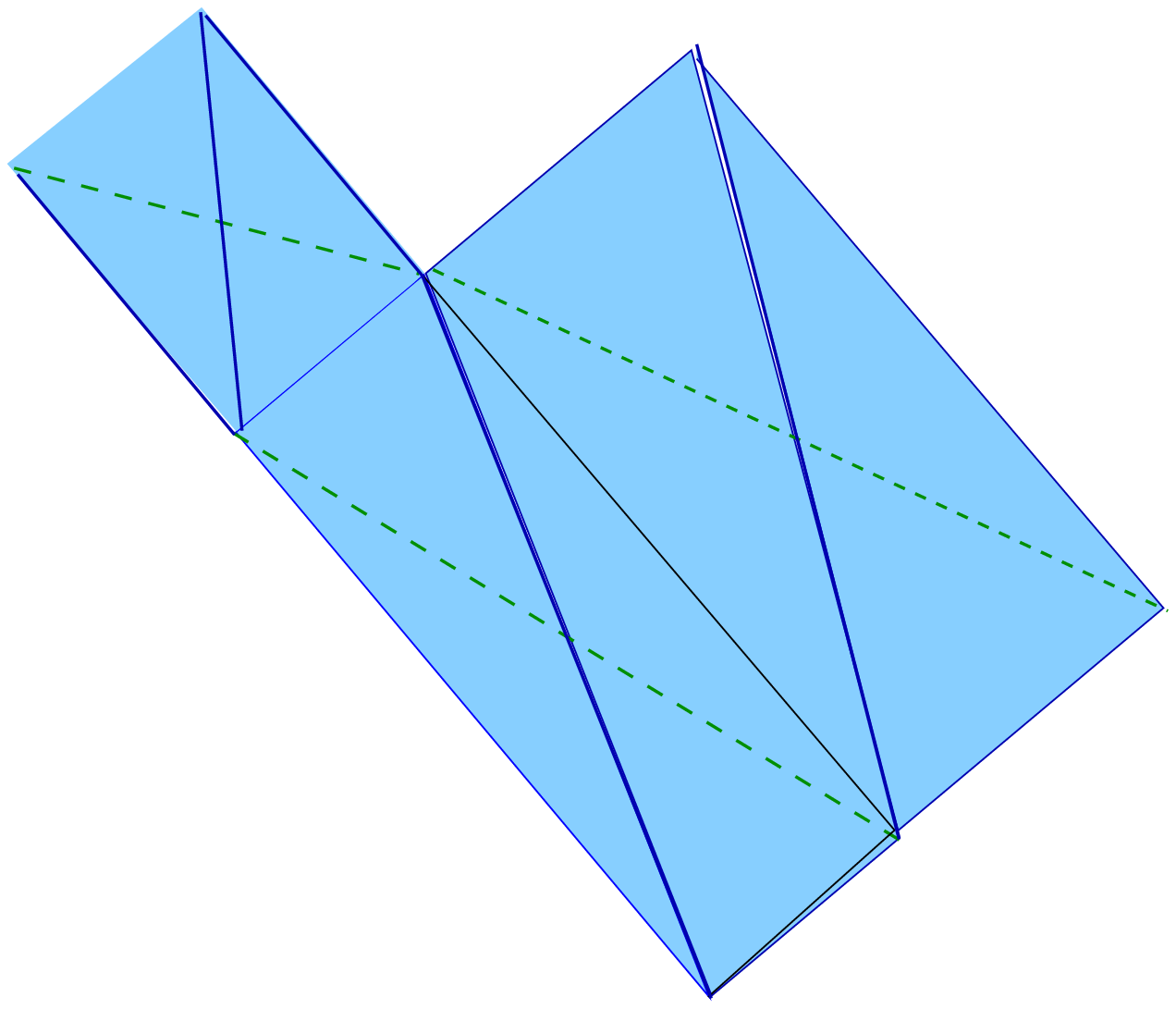}\label{73}}
\hspace{1mm}
\subfigure[ $S'$]{\includegraphics[height=0.225\textheight]{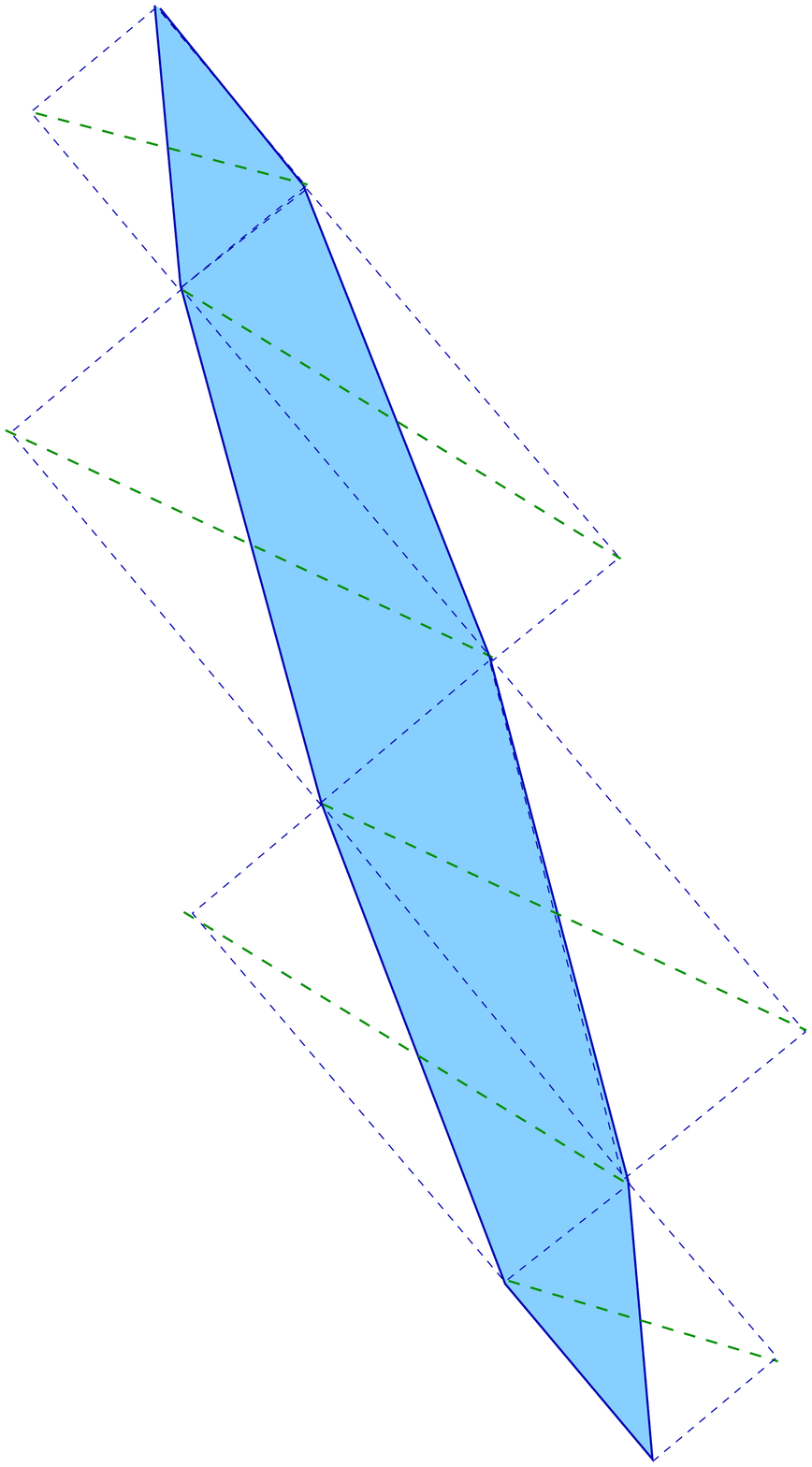}\label{74}}
\caption{L-shapes decompositions and Delaunay triangulations when  $[f]$ is on $E_0\subset \mathscr{O}$.\label{triangulationsL}}
\end{center}
\end{figure}
Consider for example the translation surfaces on the geodesic  ray $g_t ^{\theta} \cdot S_O$ with $\theta \in \Sigma_0$ and $t>0$, marked by the natural marking $\Phi_{g_t ^{\theta}}: S_O \rightarrow g_t ^{\theta} \cdot S_O$. Let $t_\theta$ be the first $t >0$ such that the triple $[\Phi_{g_t ^{\theta}}]$ 
is represented by a point of $E_0$.  If $t < t_\theta$, the Delaunay triangulation on $g_t ^{\theta} \cdot S_O$ is the pull-back by the marking of the triangulation in Figure \ref{triangles1}. This triangulation can be cut and pasted as in Figure \ref{71} to form an $L$, whose sides are in direction $\theta = 0$ and $\theta = \pi/8$ and make an obtuse angle. As $t$ increases, this obtuse angle shrinks and, as remarked above, the directions of the corresponding sides become orthogonal with respect to the flat metric exactly for  $t=t_\theta$. 
  Thus, the surface $g_t ^{\theta} \cdot S_O$ (see Figure \ref{74}) is translation equivalent to surface glued out of an $L$-shape with right angles  in Figure \ref{73} (which is obtained by cutting and pasting the affine octagon in Figure \ref{74}). Let us remark that Delaunay triangulations of translation equivalent surfaces are in correspondence under the cut and paste map.  At this point it is clear that the Delaunay triangulation of $g_t ^{\theta} \cdot S_O$ is not unique, since one can use either of the sides corresponding to the rectangle 
  diagonals in the $L$-shaped surface in Figure \ref{73} to construct a Delaunay triangulation.

The (unique) Delaunay triangulation of the surfaces $g_t ^{\theta} \cdot S_O$ with $t$ slightly bigger than $ t_\theta$  (and more precisely of all marked translation surfaces given by points in the tile $ \mathscr{F}\gamma$) are obtained by switching all sides which are diagonals in the L-shaped translation equivalent surface described above. The new sides are showed in Figures \ref{73}, \ref{74} by dotted lines.  Their pull-back to $O$ is shown in Figure \ref{71}. Let us remark that the sides of $O$ are not part of the triangulation and each triangle is obtained by gluing two of the shown triangles along a side of $O$ (as the numbers if Figure \ref{71} indicate).  This triangulation is obtained from \ref{triangles1} by Delaunay switches of sides of the octagon $O$. 
Similarly, one can prove that all translation surfaces corresponding to points on the sides of the ideal octagon tessellation are translation equivalent to translation surfaces obtained by gluing a right-angled $L$  and hence admit more than one Delaunay triangulation.

From the previous descriptions it is clear that the Delaunay triangulation changes by switches of sides of the octagon only when one crosses the boundary of the ideal octagon $\mathscr{O}$. This is the reason why in \S\ref{renormcutseq} we will consider the cutting sequence of the geodesic rays (\ref{raydef})  only  with respect to the sides of the ideal octagon tessellation. 

\paragraph{The dual tree.}
We now define a tree dual to the ideal octagon tessellation. This tree is similar in spirit to the spine defined by Smillie and Weiss \cite{SW:cha} but in this case the dual tree is not actually equal to the spine. Paths in this tree  will prove helpful in visualizing and describing the possible sequences of renormalization moves. 
\begin{figure}
\centering
\includegraphics[width=.6\textwidth]{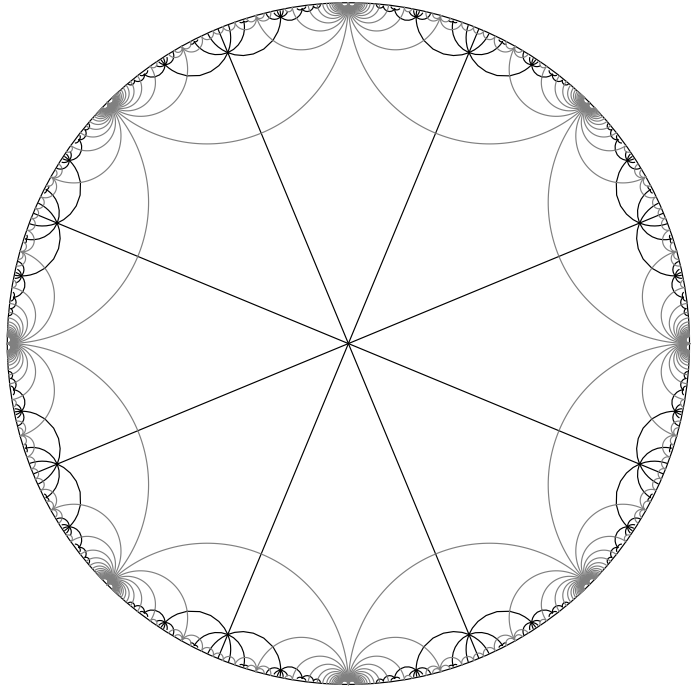}
\caption{\label{dualtree} The tessellation by ideal octagons (in gray) and the dual tree $\mathscr{T}$ (in black).}
\end{figure}
Consider the graph in the hyperbolic plane which has a vertex at the center of each ideal octagon and has an edge connecting centers of two ideal octagons when the octagons share a common side. The graph can be
embedded in $\Disk$, so that each vertex is the center of an ideal octagon and each edge is realized by a hyperbolic geodesics connecting the centers of the ideal octagons, as shown by the black lines in Figure \ref{dualtree}. We let $\mathscr{T}$ denote  this embedded graph.  
The following property is a consequence of the fact that the graph is  \emph{dual} to the ideal octagon tessellation.
\begin{rem}\label{tree} The graph $\mathscr{T}$ is a regular \emph{tree}, with eight branches at each vertex.
\end{rem}

\subsection{Renormalization and cutting sequences of Teichm{\"u}ller geodesics.}\label{renormcutseq}
One way to construct a renormalization scheme to study linear flows in a direction $\theta$ is to use a discrete sequence of elements of the Veech group, which approximate the Teichm{\"u}ller geodesic ray which contracts the direction $\theta$. 
The idea behind our renormalization is that as one flows along the geodesic, one crosses different iso-Delaunay tiles. The successive derivations of the cutting sequence will turn out to correspond to the symbolic coding of the same trajectory with respect to sides which belong to affine triangulations on $S_O$ which are pull back of  Delaunay triangulations as one moves along the geodesics. 


Let $\theta$ be a fixed direction, that we think of as the direction of a trajectory $\tau $ on  $S_O$. Recall that we denote by $\rho_\theta$ the matrix corresponding to counterclockwise rotation by $\theta$ and by $g_t ^{\theta}:=  \rho_{\frac{\pi}{2}-\theta}^{-1} \, g_{t}\ \rho_{\frac{\pi}{2}-\theta} $ a $1$-parameter subgroup  conjugate to the geodesic flow whose linear action on  $S_O$, for  $t>0$, contracts the direction $\theta$ and expands the perpendicular direction. Let us therefore consider the \emph{Teichm{\"u}ller geodesic ray}
\be\label{raydef}
\tilde{r}_\theta : = \{ g_t ^{\theta} \cdot S_O \}_{t\geq 0}, 
\ee\noindent
which, using the identification of ${\cal{\tilde M}}_A(S)$ with $T_1 \mathbb{D}$ explained in \S\ref{Teichdisksec}, corresponds to a geodesic ray in $T_1 \mathbb{D}$.  The projection $r_\theta$ of the Teichm{\"u}ller ray $\tilde{r}_\theta$ to   $\Disk$ is a half ray, starting at the center $0 \in \mathbb{D}$ and converging to the point on $\partial \mathbb{D}$ representing the linear functional given by the row vector $\begin{pmatrix}\cos( \frac{\pi}{2}-\theta)&  - \sin(\frac{\pi}{2}- \theta))\end{pmatrix} =\begin{pmatrix} \sin \theta&  - \cos \theta\end{pmatrix}$. Thus, according to the conventions in the previous section, one can check that the ray $r_\theta$ in $\Disk$  is the ray converging to the point on
 $e^{2\pi i (\pi+ 2 \theta)} \in \partial \mathbb{D}$. In particular, $r_0$ is the ray in $\mathbb{D}$ obtained intersecting   the negative real axes in $\mathbb{C}$ with $\mathbb{D}$ and  $r_\theta$ is the ray that makes  an angle $2\theta$ (measured clockwise) with the ray $r_0$. Let identify  $\mathbb{D}$ with $\mathbb{H}$ by $\phi$ (see \S\ref{discsec}) and $\partial \mathbb{D}$ with $\partial \mathbb{H} = \mathbb{R}$ by extending $\phi$ by continuity. If $x \in \mathbb{R}$ is the coordinate for $\partial \mathbb{H}$ obtained using the chart $\phi_1$ (see \S\ref{discsec}), one can check the following.
\begin{rem}\label{uthetacorrespondence}
The ray $r_{\theta}$ has endpoint
$ x(\theta) = -\frac{1}{\cot \theta}$. Moreover, for any $0\leq i \leq 7$, $r_\theta$ crosses the side $E_i = \nu_i  E_0$  of the ideal octagon $\mathscr{O}$  if and only if $\theta \in \Sigma_i = \nu_i^{-1} \Sigma_0$.
\end{rem}

\paragraph{Combinatorial geodesics.}
Let us explain how to associate to the geodesic path $r_\theta$ 
 a path $p_\theta$ in the tree $\mathscr{T}$,  which we call the \emph{combinatorial geodesic} approximating $r_\theta$.  We say that $\theta$ is a \emph{cuspidal direction} if the ray $r_{\theta}$ converges to a vertex of an ideal triangle. This is equivalent to saying that the corresponding flow on $S_O$ consists of periodic trajectories\footnote{The proof of this fact follows by combining Proposition 2.3.2 in \cite{SU:sym} with Proposition \ref{CFandcuttseq} here and it is proved below, see the paragraph following Proposition \ref{CFandcuttseq}.}. 
 Assume first that $\theta$ is a not cuspidal direction. 
 Then $r_\theta$ crosses an infinite sequence of sides of ideal octagons of the tessellation. 
In this case, the associated combinatorial geodesic $p_\theta$ on $\mathscr{T}$ is a continuous semi-infinite path on $\mathscr{T}$ which starts at $\underline{0}$ and goes, in order, through the edges of $\mathscr{T}$ which are transversal to the ideal octagon sides crossed by  the geodesics.
If $\theta$ is a cuspidal direction then $r_\theta$ crosses only a finite sequence of sides. In this case we associate to $r_{\theta}$ a finite path, which ends with the edge $e'$ transversal to the last ideal octagon $\mathscr{O}'$ crossed. We comment below on  a variation of this convention which associates two infinite paths to each ray in a cuspidal direction instead. 


\subsection{Teichm\"uller cutting sequences.}\label{Teichcuttseqsec}
In this section we will describe a coding of hyperbolic geodesics in the spirit of the Markov coding described by Series in \cite{Se:geoMc, Se:sym}.  We will call the coding assigned to a given geodesic its \emph{Teichm\"uller cutting sequence}.
The first step in constructing this coding is labeling the edges of the ideal octagon tessellation. In order to do this we introduce a subgroup of $V_P(S_O)$.
 Define the group $V_{\mathscr{O}}$ to be the group generated by hyperbolic reflections in the sides of $\mathscr{O}$.
The element $\gamma$ is the hyperbolic reflection which, when  acting on the right, fixes side $E_0$, the other reflections generating $V_{\mathscr{O}}$ are $\gamma_i:= \nu_{i} ^{-1} \gamma \nu_i$, $i=1, \dots, 7$, so that the right action of $\gamma_i$ is the hyperbolic  reflection fixing side $E_i$. 
It follows from the Poincar\'e Polyhedron Theorem (see \cite{RA:fou}) that the ideal octagon $\mathscr{O}$ is a fundamental domain for the subgroup $V_{\mathscr{O}}$. Since this fundamental domain is built from eight copies of the fundamental domain of $V_P(S_O)$ it follows that $V_{\mathscr{O}}$ is a subgroup of index eight. Since $\mathscr{O}$ is a fundamental domain every point in $\Disk$ is equivalent to a point in $\mathscr{O}$ by means of an element of $V_{\mathscr{O}}$. For a general fundamental domain of a Fuchsian group it could be the case that a point would be equivalent to more than one point on the boundary of a fundamental domain. The different points would be related by the side paring elements in the group as discussed in \cite{RA:fou}, \S6.6. In our case since each side pairing element pairs a side with itself, that is to say that since the group $V_{\mathscr{O}}$ is generated by reflections, every point is equivalent to a unique point $\mathscr{O}$. A fundamental domain with this property is called a \emph{strict} fundamental domain.

Now we can label the edges in the tessellation by elements of $\{0, \dots, 7\}$. For $i=0, \dots, 7$, we label the side $E_i$ of $\mathscr{O}$  by $i$  and then transport the labeling to all other sides  by the action of $V_{\mathscr{O}}$ using the fact that each point on a side of the tessellation is equivalent to a unique point in $\partial\mathscr{O}$. The labeling has the following property. If the side $E$ is shared by two neighboring ideal octagon tiles  and and $E$ carries the label $i$ for some $0\leq i \leq 7$, then there exist $n\in \mathbb{N}$ and $s_i \in \{0, \dots, 7\}$ for $i=0, \dots, n$, such that  one tile is $ \mathscr{O}\gamma_{s_n} \dots   \gamma_{s_1}\gamma_{s_0} $  and   
the other is $  \mathscr{O}\gamma_i  \gamma_{s_n} \dots  \gamma_{s_1} \gamma_{s_0}$. 

 Following Series (see \cite{Se:geoMc}, \cite{Se:sym} and the references therein), one can associate to any geodesic a \emph{cutting sequence}. 
We define the Teichm\"uller cutting sequence  associated to a geodesic to be the sequence of labels of sides crossed in the ideal octagon tessellation. In particular, if $r_{\theta}$ is a geodesic ray, we associate to it a sequence $c(r_{\theta})$ of labels in $\{0, \dots, 7\}$. If ${\theta}$ is a cuspidal direction, the sequence of sides crossed and hence   $c(r_{\theta})$ is finite, otherwise  $c(r_{\theta}) \in \{0, \dots, 7\}^{\mathbb{N}}$. As geodesics do not backtrack, in the sequence $c(r_{\theta})$ the same symbol $i$ never appears twice in a row. One can easily  see that this is the only restriction, thus a sequence $\{c_i \}_{i\in \mathbb{N}} \in \{0, \dots, 7\}^{\mathbb{N}}$ is the cutting sequence of a ray $r_{\theta}$ if and only if,  for all $i\in \mathbb{N}$,  $c_i=j$ for some $0\leq j \leq 7$  implies $c_{i+1} \neq j$.

Since edges of the dual tree $\mathscr{T}$ are in one to one correspondence with  sides of ideal octagons (each edge being transversal to one ideal octagon side), one obtains in this way a labeling of the edges of the tree $\mathscr{T}$. The labeling of the edges is shown in Figure \ref{Seriestreefig}. The sequence of labels of the edges of the path $p_\theta$ associated to the ray $r_{\theta}$ gives again the cutting sequence $c(r_{\theta})$. 

\begin{center}
\begin{figure}
\subfigure[Labelling of the tree using reflections in $V_{\mathscr{O}}$]{\includegraphics[width=0.44\textwidth]{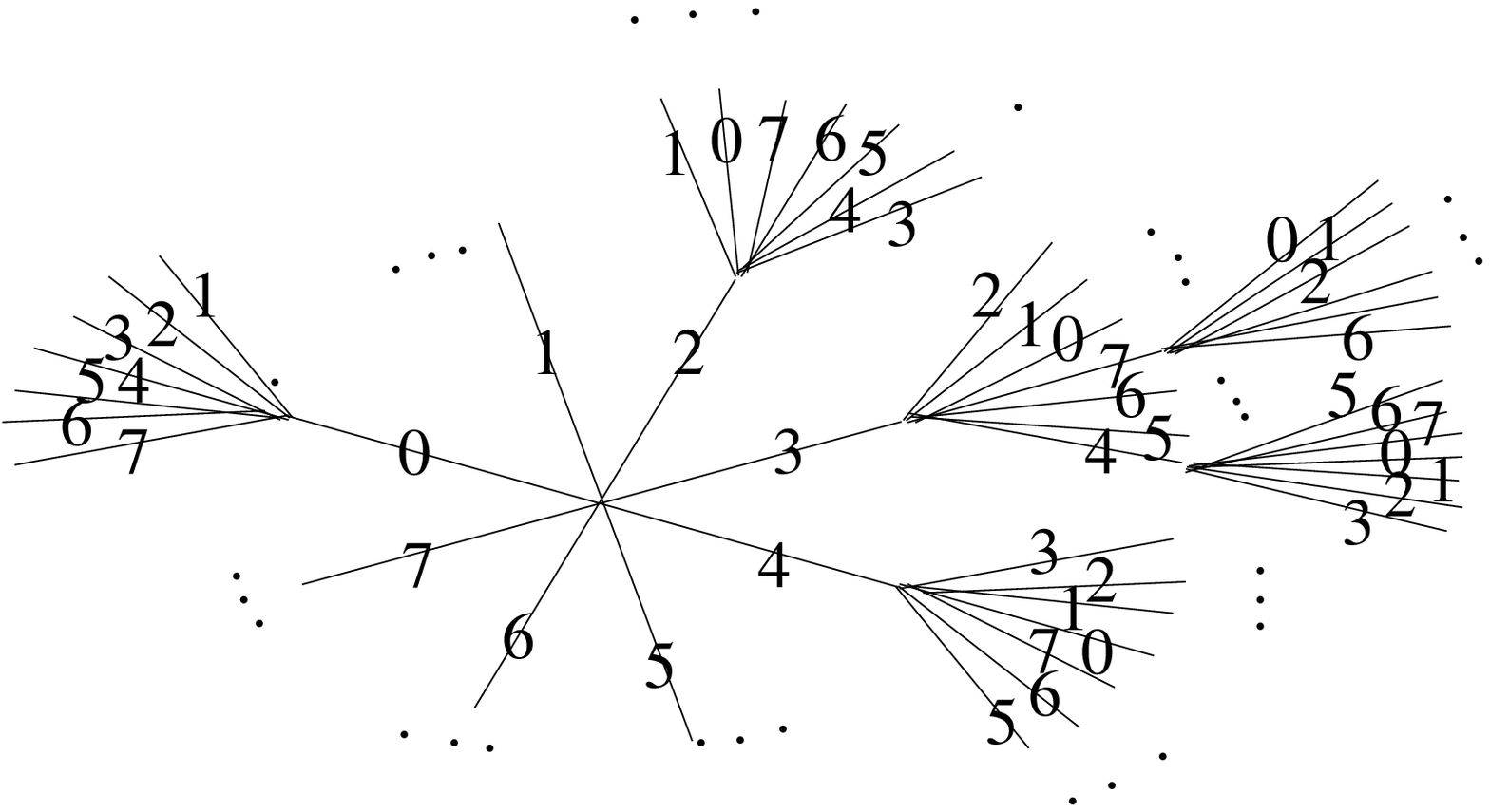}\label{Seriestreefig}}
\hspace{9mm}
\subfigure[Normalized labeling of the tree using $\gamma \nu_i$ ]{\includegraphics[width=0.44\textwidth]{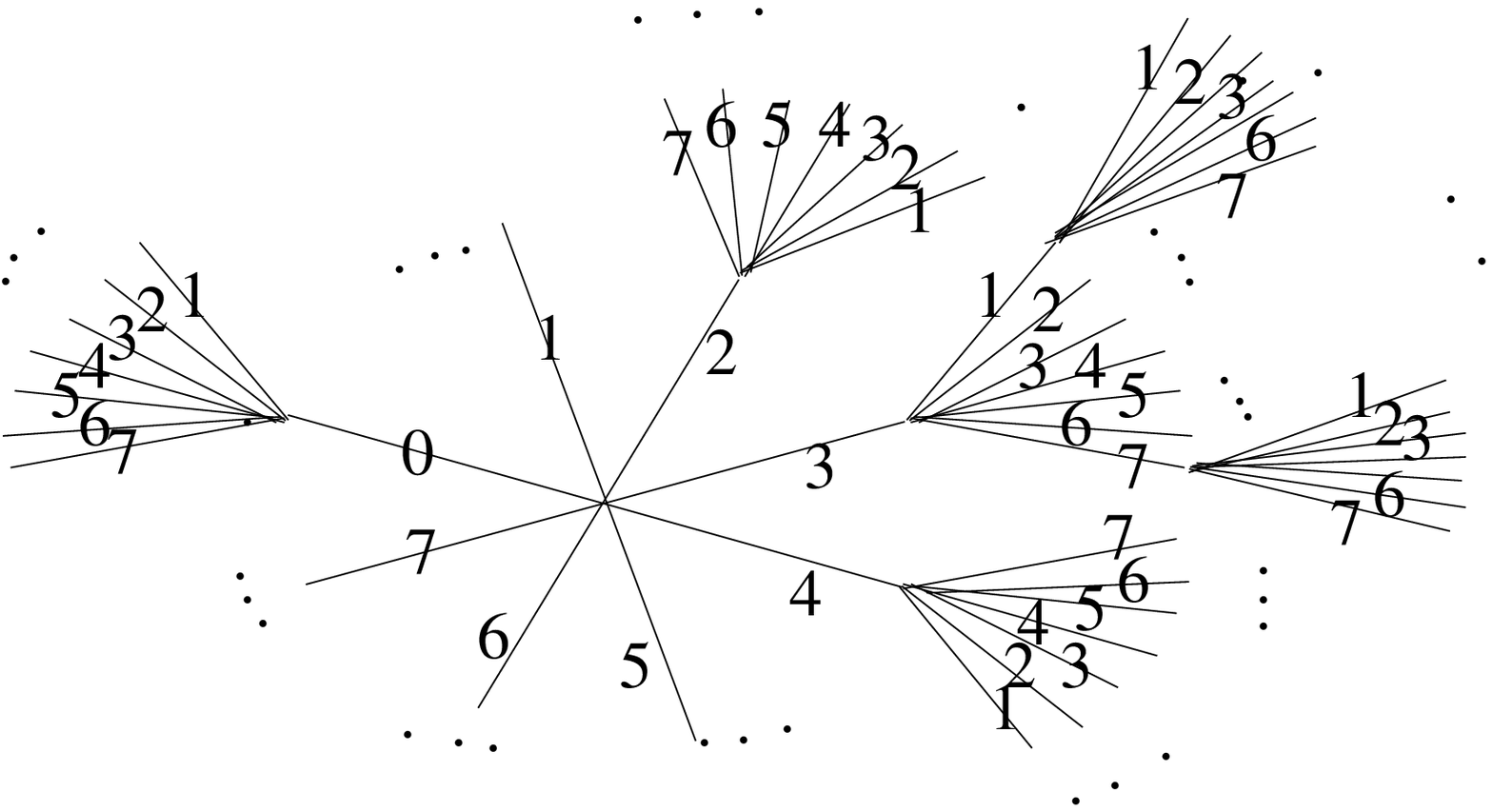}\label{treelabeling}}
\caption{The two labelings of the tree edges on a schematic picture of $\mathscr{T}$.}
\end{figure}
\end{center}

\paragraph{Teichm\"uller cutting sequences  and derived  sequences.}
If $w$ is a cutting sequence of a trajectory $\tau$ in direction $\theta$, let $r_{\theta}$ be the geodesic ray which contracts the direction $\theta$ given in (\ref{raydef}). We are going to use the  cutting sequence $c(r_{\theta})$  of $r_{\theta}$ to give a geometric interpretation of the derived sequences $w^{(k)}$. 
If the cutting sequence $c(r_{\theta})$ starts with $s_0, s_1, \dots, s_k$, then for each $k\geq 1$, the element $\gamma^{(k)}:= \gamma_{s_0} \dots \gamma_{s_{k-2}} \gamma_{s_{k-1}} $ is the element of $V(S)$ which acts on the right on  $\Disk$ by mapping the $k^{th}$ ideal octagon tile crossed by $r_{\theta}$ back to $\mathscr{O}$.  
Since $\gamma^{(k)} \in V(S)$, there is an affine automorphism $\Psi_{\gamma^{(k)}}: S_O \rightarrow S_O$ whose derivative is  $\gamma^{(k)}$. Moreover, by Remark \ref{trivialkernel}, since by Lemma \ref{Veechprop} the kernel of the Veech homomorphism is trivial, this affine automorphis is unique and can be obtained composing the canonical map  $ \Phi_{\gamma^{(k)}} : S_O \rightarrow  {\gamma^{(k)} \cdot S_O } $ with a uniquely defined translation equivalence $\Upsilon^{(k)}  :{\gamma^{(k)} \cdot S_O }  \rightarrow O$.

Consider the image $\Psi_{\gamma^{(k)}} O \subset S_O$ of the standard octagon $O \subset S_O$. Since $\Psi_{\gamma^{(k)}} = \Upsilon^{(k)} \Phi_{\gamma^{(k)}}$, $\Psi_{\gamma^{(k)}} O $ is obtained by first linearly deforming $O$ to the affine octagon $O^{(k)}: =\gamma^{(k)}  O$ and then cutting it and pasting back to $O$ according to $\Upsilon^{(k)}$ (see for example Figure \ref{shearedoctagon1}, where $k=1$ and $\gamma^{(k)} = \gamma$).  If a side of $O$ is labeled by $L \in \{ A,B, C, D\}$, let us label by $L$ also its image under $\Psi_{\gamma^{(k)}}$. This gives  a labeling of the sides of $\Psi_{\gamma^{(k)}} O $ by  $\{ A,B, C, D\}$  which we call the \emph{labeling induced by} $\Psi_{\gamma^{(k)}}$.  The connection between Teicm\"uller cutting sequences and derived cutting sequences is the following.  

\begin{prop}\label{derivedseq}
The $k^{th}$ derived sequence $w_k$ of the cutting sequence $w$ of a trajectory on $S_O$ is the cutting sequence of the same trajectory  with respect to the sides of $\Psi_{\gamma^{(k)}} O $ with the labeling induced by $\Psi_{\gamma^{(k)}}$. 
\end{prop}
\noindent Before giving the proof of Proposition \ref{derivedseq}, let us remark that as $k$ increases the affine octagons $O^{(k) } \subset \mathbb{R}^2$ become more and more  stretched in the direction $\theta$,  meaning that the directions of the sides of $O^{(k)}$ tend to $\theta$ . This can be checked by  verifying that the sector of directions which is the image of $\RP \backslash \Sigma_{s_0}$ under $\gamma^{(k)} $  is shrinking to the point corresponding to the line in direction $\theta$. This distortion of the octagons corresponds to the fact that as $k$ increases a fixed trajectory hits the sides of  $O_k$ less often which is reflected by the fact that in deriving a sequence letters are erased.  We also remark that the sides of  $\Psi_{\gamma^{(k)}} O $ are a subset of the affine triangulation of $O$ obtained as  the pull-back of Delaunay  triangulations corresponding to the $k^{th}$ ideal octagon tile entered by $r_{\theta}$. 

\begin{proofof}{Proposition }{derivedseq}  
To prove that $w^{(k)}$ is the cutting sequence of $\tau$ with respect to $\Psi_{\gamma^{(k)}} O $, one can equivalently apply the affine diffeomorphism $\Psi_{\gamma^{(k)}}^{-1}$ and prove that $w^{(k)}$ is the cutting sequence of $ \Psi_{\gamma^{(k)}}^{-1} \tau$ with respect to $  O$. Let us show this by induction on $k$. Set $\tau^{(0)}:=\tau$ and  for $k>0$ set $\tau^{(k)}:= \Psi_{\gamma^{(k)}}^{-1}  \tau 
= \Psi_{\gamma_{s_{k-1}}} \dots  \Psi_{\gamma_{s_{0}}} \tau $ (note that $\gamma_i$ are involutions). The base of the induction is simply $w=c(\tau)$.   For any $k\geq 0$, assume that $w^{(k)}$ is the cutting sequence of $ \tau^{(k)}$  with respect to $O$.  Since the $k^{th}$ element of the cutting sequence is $s_k$, the direction of $ \tau^{(k)}$ belongs to $\Sigma_{s_{k}}$. 
Thus, $\nu_{s_{k}} \tau^{(k)} $ has direction in $\Sigma_0$ and by (\ref{relabelingeq}) its cutting sequence is $\pi_{s_{k}}\cdot w^{(k)}$. 
Thus, by (\ref{derivedtraj}), 
 $(\pi_{s_{k}}\cdot w^{(k)})'$ is the cutting sequence of  $\Psi_\gamma \nu_{s_{k}}  \tau^{(k)}$. Thus, again by (\ref{relabelingeq}), if we act by $\nu_{s_{k}}^{-1}$ and remark that $\Psi_{\gamma_{s_k}} = \nu_{s_{k}}^{-1} \Psi_{\gamma} \nu_{s_{k}} $, the cutting sequence of  $ \tau^{(k+1)} = \Psi_{\gamma_{s_k}} \, \tau^{(k)} = \nu_{s_{k}}^{-1}\Psi_{\gamma} \nu_{s_{k}}   \tau^{(k)} $ is $\pi_{s_{k}}^{-1} \cdot (\pi_{s_{k}}\cdot w^{(k)})'$ which  is equal to $(w^{(k)})' = w^{(k+1)}$ because the action by permutations commutes with derivation (i.~e.~$(\pi \cdot w )' =\pi \cdot w' $). 
\end{proofof}

\subsection{Normalized Teichm\"uller cutting sequences.}
The Teichm\"uller cutting sequences above are invariant under the  action $V_\mathscr{O}$ on geodesics but they are not invariant under the action of $V_P(S_O)$.  Let us describe now a different (and invariant) way of assigning to Teichm\"uller geodesics a cutting sequence, that we call  \emph{normalized} cutting sequence.  In \S\ref{crosssec} below we show  that normalized Teichm\"uller  cutting sequences have a natural  interpretation as a coding of  the first return map of the geodesic flow on the Teichm\"uller orbifold represented by the fundamental domain $\mathscr{F}$ to the side $E_0$ of  $\mathscr{F}$.  As Proposition \ref{derivedseq} shows that Teichm\"uller cutting sequences help understand derived sequences $w^{(k)}$ of cutting sequences of linear trajectories, we show in Proposition \ref{wknormalized} below that \emph{normalized} Teichm\"uller cutting sequences help understand the cutting sequences $w_k$ in the renormalization scheme explained in \S \ref{renormschemessec}. 


Let us consider the elements $\nu_k ^{-1} \gamma  $ of the Veech group, for $k=0, \dots, 7$, 
and describe their action on $\Disk$. 
 Let $\mathscr{O}_k  $ be the ideal octagon obtained by reflecting $\mathscr{O}$ through the side $E_k$.   The element $\nu_k^{-1} \gamma$, acting on the right (first by $\nu_k^{-1}$, then by $\gamma$), maps the ideal octagon $\mathscr{O}_k$ back to the octagon $\mathscr{O}$, sending the side $E_k = E_0 \nu_k $ to $E_0$, or in other words, $\mathscr{O} = \mathscr{O}_k \nu_k ^{-1}\gamma$ and  $E_0 = E_k \nu_k^{-1} \gamma$.
  
 One can associate to any (segment of) a hyperbolic geodesic $g$ a \emph{normalized} cutting sequence, that we denote by $\bar c(g)$ as follows. Entries $\bar c(g)_i$ of the sequence $\bar c(g)$ are associated to exiting sides of ideal octagon tiles crossed by $g$, or in other words we associate a label to the side $E''$ 
if $g$ enters a tile $\mathscr{O}'$ though the side $E'$ and leaves it through the side $E''$. 
 Let  $\nu' \in V_P(S_O)$ be  the unique element which, acting on the right, sends  $\mathscr{O'}$ to $\mathscr{O}$ and  the entering side $E'$ to $E_0$. The existence and uniqueness of such element follows by Poincar\'e Polyhedron Theorem and the fact that $\mathscr{F}$ is a strict fundamental domain for $V_P(S_O)$ (as discussed in \S\ref{Teichcuttseqsec}).
 If the exiting side $E''$ is mapped by $\nu' $ to $E_k$, than we give to  $E''$  label $k$. By construction $k$ is in $ \{ 1, \dots, 7\}$ since $g$ cannot enter and leave through the same side. If both the forward and backward  endpoints of $g$ are not ideal vertexes,  $g$ crosses infinitely many tiles and  $\bar c(g) $ is a bi-infinite sequence in $ \{ 1, \dots, 7\}^{\mathbb{Z}}$, defined up to shift. If we are given a reference point $p$ on $g$, we set $\bar c(g)_0$ to be the label of the exiting side of the tile to which $p$ belongs. 
 
This coding of geodesics is by construction invariant under the action of the Veech group. 
One can obtain the normalized coding $\bar c (g)$ from the Teichm\"uller cutting sequence $c(g)$ using the following Lemma. 
 \begin{lemma}\label{dictionary}
Assume that in the Teichm\"uller cutting sequence $c(g)$ we have $c(g)_{l-1} = i  $ and  $c(g)_{l} = j  $  for some $l\in \mathbb{Z}$. Then in the normalized cutting sequence we have  $\bar c(g)_{l} = k  $ where $k $ is the unique $k \in \{1, \dots, 7 \}$ such that $ \nu_j \nu_i ^{-1}  = \nu_k$ in $D_4 = D_8/\{\pm I\}$. 
\end{lemma}
\begin{proof}
Since ${c(g)}_{l-1} = i  $ and  $c(g)_{l} = j  $, we know that $g$ crosses three consecutive tiles $\mathscr{O}',  \mathscr{O}'',\mathscr{O}''$ where  for some $\nu \in V_{\mathscr{O}}$ we have $\mathscr{O}'=  \mathscr{O}\nu$, $\mathscr{O}''=  \mathscr{O}\gamma_i\nu $ and  $\mathscr{O}'''=  \mathscr{O}\gamma_j\gamma_i \nu $. Let $E'$ be the common side of $\mathscr{O}'$ and  $\mathscr{O}''$ and $E''$ be the common side of $\mathscr{O}''$ and  $\mathscr{O}'''$.  Let us act on the right by $ \nu^{-1} \gamma_i^{-1} \in V_{\mathscr{O}}$, which maps $\mathscr{O}''$ to $\mathscr{O}$. 
Since $g$ crosses in order $ \mathscr{O}'$, $\mathscr{O}''$ and $\mathscr{O}'''$, which are mapped by  $\nu^{-1}\gamma_i^{-1} $ to $\mathscr{O}$, $\mathscr{O}\gamma_i^{-1}$ and $\mathscr{O}\gamma_j$ respectively,  we know that $  g \nu^{-1}\gamma_i^{-1}$ enters $\mathscr{O}$ through the side shared with $\mathscr{O}\gamma_i^{-1} $, which is $E_i$, and exit through the side $E_j$ shared with $ \mathscr{O}\gamma_j$. Thus, if we now act on the right by $\nu_i^{-1}$ which maps $E_i$  to $E_0$, we see that globally the right action of the element $ \nu_i ^{-1}\nu^{-1} \gamma_i^{-1} $ maps $\mathscr{O}''$ to $\mathscr{O}$, the entering side $E'$ to $E_0$ and the exiting side $E''$ to $E_j \nu_i ^{-1}$. The equation $\nu_j \nu_i ^{-1}  = \nu_k$ and (\ref{sidesrightaction}) imply that $ E_j \nu_i^{-1} = E_0 \nu_j \nu_i ^{-1} =E_0 \nu_k = E_k$. By definition of the normalized coding, this gives $\bar c(g)_{l} = k  $. 

\end{proof}

If we have a geodesics ray $r_{\theta}$, the first symbol $\bar c(r_{\theta})_0$ of $\bar c(r_{\theta})$ is not defined by the rule we have described since the initial tile is not crossed completely: we set by convention $\bar c(r_{\theta})_0 =  k \in  \{0, 1, \dots, 7\}$ iff  $r_{\theta}$ leaves $\mathscr{O} $ through the side $E_k$. The following elements of  $\bar c(r_{\theta})$ are associated to exiting sides of tiles crossed by $r_{\theta}$ and belong to  $\{ 1, \dots, 7\}$.  Thus, the $\bar c(r_{\theta})$ is a one-sided sequence, finite if $\theta$ is cuspidal, infinite otherwise.  
Using again the fact that edges of the tree $\mathscr{T}$ are dual to sides of ideal octagons, this labeling of sides induces a labeling of the tree, which is shown in Figure \ref{treelabeling}.  
By construction this labelling has  the property that the normalized cutting sequence $\bar c(r_{\theta})$ is given by the sequence of labels associated to vertexes of the combinatorial geodesics $p_\theta$ approximating $r_{\theta}$. 

One can also give an intrinsic description of the labeling of the edges of the tree, without passing through Teichm\"uller cutting sequences, but using the tree structure. 
 Let us define the root of our tree to be the center of the disk $\mathbb{D}$ and say
 that the \emph{level $n$}  is composed by all vertexes which have distance $n$ from the root, where the distance here is the natural  distance on a graph which gives distance $1$ to vertices which are connected by an edge. For $n\ge 1$ we call \emph{edges of level $n$} the edges which connect a vertex of level $n-1$ with a vertex of level $n$. 
Consider the elements $ \nu_k^{-1}\gamma $, $k= 1,\dots, 7$ and consider their right action on the embedded copy of $\mathscr{T}$. One can check the following. Let $e_k$ be the edge of level $1$ which is perpendicular to the side $E_k$ of $ \mathscr{O}$ and let $ \mathscr{O}_k$ be the ideal octagon tile obtained reflecting $\mathscr{O}$ at $E_k$.
\begin{rem}\label{actionontree}
The right action of the element $ \nu_k^{-1}\gamma  $ gives a tree automorphism of $\mathscr{T}$, 
 which sends the center  of the octagon $\mathscr{O}_k$ to the root $\underline{0}$ and the oriented edge from $\underline{0}$ to the center of $\mathscr{O}_k$ to the oriented edge from the center of $\mathscr{O}_0$ to $\underline{0}$. Moreover $ \nu_k^{-1}\gamma $ maps  all the edges of level $n$ which branch out of $e_k$  to edges of level $n-1$.
\end{rem}
This observation can be used to define  a labeling of the tree edges by induction on the \emph{level} of the edges. 
Let us label by $\{0, \dots, 7\}$  the edges of level $1$ of the tree, the edge $e_k $ perpendicular to $E_k$ being labeled by $k$. If $e$ is an edge  of level $2$ branching out of $e_k$, 
 applying $ \nu_k^{-1}\gamma  $ to the right, by Lemma \ref{actionontree}, it is mapped to an edge of level $1$; let its label be $j$ iff this edge is $e_j$. We remark that $j \in \{ 1, \dots, 7\}$. Indeed the edge $e_0$ is the image by  of a level-$1$ edge $e_k$ hence the label $0$ is not assumed by any level $2$ edge. Since by Lemma \ref{actionontree} all of edges of level $2$ are mapped to edges of level $1$ by one of the renormalization moves, this allows us to iterate the same procedure to define  labels of level $3$, and higher by induction. See Figure \ref{treelabeling}. 
One can check by induction that the following holds. 
\begin{lemma}\label{actionontreepath}
Given a finite path on the tree from the root to the  center of an ideal octagon $\mathscr{O}'$, if the labels of its edges are in order $s_0, s_1, \dots, s_n$, where $s_i$ is the label of the edge of level $i$, then the  element $ \nu_{s_0}^{-1}\gamma \cdots\nu_{s_n}^{-1}  \gamma $ 
 acts on the right by mapping  the last edge $e'$ to $e_0$ and the octagon $\mathscr{O}'$ to $\mathscr{O}$.
\end{lemma}

\paragraph{The first renormalization move.}
The renormalization schemes in \S \ref{renormschemessec} are expressions of the following renormalization scheme on paths on the tree $\mathscr{T}$ (or combinatorial geodesics). 
Given a direction $\theta$, let   $p_{\theta}$ be the combinatorial geodesics approximating $r_\theta$. Let us assume for now that $\theta$ is not a cuspidal direction. By  Remark \ref{uthetacorrespondence}, $\theta \in \Sigma_k$ iff $r_\theta$ crosses the side $E_k$ of $\mathscr{O}$ or equivalently iff 
 the first label of  $\bar c(r_{\theta})$ is $k$. 
Thus, if  the vertexes of the combinatorial geodesics $p_{\theta}$ are labeled by $s_0, s_1, \dots$, let us act on the right by the renormalization element $ \nu_{s_0}^{-1} \gamma $.  The right action of  $\nu_{s_0}^{-1}\gamma  $  on $\Disk$ sends  $p_{\theta}$ to a new combinatorial geodesics, which starts from a vertex of level $1$ and passes through $\underline{0}$. If we neglect the first edge, this new combinatorial geodesics $p'$ has vertexes labeled by  $s_1, s_2,  \dots $. Thus, at the level of combinatorial geodesics labelings, this renormalization act as a shift. 
Let us consider the limit point in $\partial \Disk$  of $p'$. 
\begin{lemma}\label{limitpoint}
The limit point of the combinatorial geodesics $p'$ in $\partial \Disk$  is the endpoint of the Teich\-m\"uller ray $r_{\theta'}$, where $\theta'=F(\theta)$ and $F$ is the Farey map. 
\end{lemma}
\begin{proof}
Let us assume that $\theta \in \Sigma_i$. Then $p'$ is obtained by acting on the right on $p$ by $\nu_{i}^{-1} \gamma $. Since $p$ by construction has the same limit point than $r_{\theta}$ and the action of $\nu_{i}^{-1} \gamma$ extends by continuity to $\partial \mathbb{D}$, the limit point of $p'$ is obtained by acting on the right on the limit point of $r_{\theta}$. Let us identify  $\partial \mathbb{D}$ with $\mathbb{R}$ as in \S\ref{discsec} and let $\bar x\in \mathbb{R}$ be the endpoint. Its image $\bar x'$ by the right action of $\nu_{i}^{-1} \gamma = \left( \begin{smallmatrix} a_i & b_i \\ c_i & d_i  \end{smallmatrix} \right)$ is  $\bar x' = \frac{a_i \bar x  + c_i}{ b_i \bar x  +d_i}$. By Remark \ref{uthetacorrespondence}, this is the endpoint of the ray  $r_{\theta'}$ where $\cot \theta ' = -1/(\bar x ') = - \frac{ b_i \bar x  + d_i}{a_i \bar x  + c_i}  $ and since $\bar x = 1 / \cot \theta$ again by Remark \ref{uthetacorrespondence}, we get  $\cot \theta ' =   \frac{  d_i  \cot \theta - b_i   }{  -  c_i \cot \theta+ a_i}  $. This is exactly the left action by linear fractional transformation of the inverse $\left( \begin{smallmatrix} d_i& - b_i \\ -c_i & a_i   \end{smallmatrix} \right) = \gamma \nu_i $. This shows exactly that $\theta ' = F(\theta)$, by definition of the octagon Farey map (see  \S\ref{Fareymapsec}). 
\end{proof}
\noindent Iterating the renormalization move on  the combinatorial geodesics  and Lemma \ref{limitpoint} one can prove the following. 
\begin{prop}\label{CFandcuttseq}
Let  $\theta = [s_0; s_1, \dots, s_n, \dots ] $. The entries $\{s_k\}_{k\in \mathbb{N}}$ of the octagon additive continued fraction give the normalized  cutting sequence $\bar c(r_\theta)$ of the Teich\-m\"uller geodesics  ray $r_{\theta}$.   
\end{prop}
\noindent Proposition \ref{CFandcuttseq} actually holds for all $\theta$ 
 if we adopt a different convention for the combinatorial geodesics associated to a cuspidal direction. Instead of associating to a cuspidal direction $\theta$ a finite combinatorial geodesics $r_{\theta}$ made of $n$  edges  labeled in order by $s_1, \dots, s_n \in \{1, \dots, 7\}$, we can associate to $\theta$ two infinite paths. If $s_n $ is even (respectively odd), we associate the infinite path which after the $n^{th}$ edge goes only though edges labeled by $1$ (respectively $7$) and the path which shares the first $n-1$ edges, then goes through  the edge labeled by $s_{n}+1 $ (modulo $7$) and then only through edges labeled by $1$ (respectively $7$)\footnote{These two paths corresponds exactly to the two possible additive octagon continued fraction expansions of a terminating direction, described in Lemma 2.2.17 in \cite{SU:sym}.}. One can see in Figure \ref{treelabeling} that these two infinite paths both limit to the same point on $\partial \mathbb{D}$. 
With this convention Proposition \ref{CFandcuttseq} holds for all $\theta$ and shows that cuspidal directions on the Teichm\"uller disk are in one-to-one correspondence with terminating directions, i.~e.~directions whose octagon additive continued fraction expansion is eventually 1 or 7.   Moreover, cuspidal directions correspond exactly to directions $\theta$ such that all trajectories in direction $\theta$ are periodic sequences (see Proposition 2.3.2 in \cite{SU:sym}). 

As a consequence of Proposition \ref{CFandcuttseq}, we have the following strengthening of Lemma \ref{limitpoint}.
\begin{cor}
The new combinatorial geodesics $p'$  is the combinatorial geodesic corresponding to the Teichm\"uller ray $r_{\theta'}$, where $\theta'=F(\theta)$. 
\end{cor}

\paragraph{Normalized  cutting sequences and combinatorial renormalization.}
Let  $w$ be the cutting sequence of a trajectory $\tau$ in direction $\theta$ and assume that $w$ is non-periodic. As in \S\ref{Teichcuttseqsec} we made a connection between the Teichm\"uller cutting sequence of the ray $r_\theta$ and the derived sequences $w^{(k)}$, here we use the \emph{normalized} cutting sequence $\bar c(r_\theta)$ to give a geometric interpretation of the cutting sequences $w_k$, obtained by normalizing and deriving by (\ref{renormalizedseq}), according to the combinatorial renormalization algorithm described in \S\ref{combrenormsec}. 

If the labels of the edges of the combinatorial geodesics $p_{\theta}$ approximating the ray $r_{\theta}$ are $s_0, s_1, \dots$, let $\nu(k) : = \nu_{s_0 }^{-1}\gamma \nu_{s_{1}}^{-1} \gamma\dots \nu_{s_{k-1} }^{-1}\gamma$ be the element of $V(S_O)$ which maps the $k^{th}$ edge of $p_{\theta}$ to $e_0$ and the $k^{th}$ octagon tile entered by $r_{\theta}$ to $\mathscr{O}$ (see Lemma \ref{actionontreepath}). 
Since $\nu(k) \in V(S_O)$, there exists an affine automorphism $\Psi_{\nu(k)}: S_O \rightarrow S_O$ whose linear part is   $\nu(k)$ and such an element is unique (by  Lemma \ref{Veechprop} and Remark \ref{trivialkernel}).  
Let us consider the image $\Psi_{\nu(k)} O \subset S_O$ of the standard octagon $O \subset S_O$. Let us recall that $\Psi_{\nu(k)} O $ inherits a labeling by  $\{ A,B, C, D\}$ induced by $\Psi_{\nu(k)}$, the image of a side of  $O$ being labeled by the same letter of the side of $O$.  
 \begin{prop}\label{wknormalized}
 The cutting sequences $w_k$ is the cutting sequence of the same trajectory $\tau$ with respect to the sides of the affine octagon $\Psi_{\nu(k)} O$ with the labeling induced by  $\Psi_{\nu(k)}$.
 \end{prop}
\noindent The proof is given below. Since $\Psi_{\nu(k)}$ can be written as $\Psi_{\nu(k)} = \Upsilon_{k} \Phi_{\nu(k)}$,
 where $ \Phi_{\nu(k)} : S_O \rightarrow  {\gamma^{(k)} \cdot S_O } $ and $D \Upsilon_{k} = Id$ is uniquely defined,  $\Psi_{\nu(k)} O$ is the image of the affine octagon $O^{(k)}: =\nu(k)  O$ under the cut and paste map given by $\Upsilon_{k}$. The affine octagons $O^{(k) } \subset \mathbb{R}^2$ are more and more  stretched in the direction $\theta$,  meaning that the directions of the sides of $O_k$ tend to $\theta$ as $k$ increases and this show geometrically that  recoding $\tau $ with respect to these new sides provides a more efficient coding. 
 Let us remark that $O_k$ and the affine octagons $O^{(k)}$ in \S\ref{Teichcuttseqsec} are indeed the same affine octagons in  $\mathbb{R}^2$, 
but the labeling of the sides induced by $\Psi_{\nu(k)}$ and $\Psi_{\gamma^{(k)}}$ respectively are different. Correspondingly, $w_k$ (the $k^{th}$ derived and renormalized sequence in the combinatorial renormalization scheme) and $w^{(k)}$ (the $k^{th}$ derived sequence) differ only by a permutation of the letters. 
\begin{proofof}{Proposition}{wknormalized}
By Proposition \ref{relationsrenormalizations} we know that $w_k$ is the cutting sequence of the trajectory $\tau_k$ defined in \S\ref{renormschemessec}. So, applying $\Psi_{\nu(k)}$, which is an affine automorphism,  $w_k$ is also the cutting sequence of $ \Psi_{\nu(k)} \tau_k$ with respect to $\Psi_{\nu(k)} O $. It suffices to show that $ \Psi_{\nu(k)} \tau_k = \tau$. Equivalently, it is enough to prove that  $\tau_k = \Psi_{\nu(k)}^{-1} \tau$. This follows by induction after recalling that $\tau_k = n(\tau_{k-1})'$ (see \S\ref{renormschemessec}), using that, since the direction on $n(\tau_{k-1})$ is in $\Sigma_0$  we have $n(\tau_{k-1})' = \Psi_{\gamma} n(\tau_{k-1})$ by (\ref{derivedtraj})  and using  that $n(\tau_k) = \nu_{s_{k-1}} \tau_{k-1}$ since the direction of $\tau_{k-1}$ belongs to  sector $s_{k-1}$ by  Proposition \ref{relationsrenormalizations}.  
\end{proofof}

\section{Natural extension and invariant measure for $F$}\label{natextsec}

\subsection{Natural extension of the octagon Farey map}
Recall that we  denote by $\bar \Sigma:= \bar \Sigma_1 \cup \dots \cup \bar \Sigma_7$ the sector of $\RP$ given in angle coordinates by  $[\pi/8, \pi]$. In this section we will consider the restriction of the octagon Farey map $F$ to $\Sigma$, which is the invariant set of $F$. 
 Let us define a map $\hat{F}$ on $\bar \Sigma \times \bar \Sigma_0$. We will see later that this map realizes the natural extension of $F|_{\Sigma}$ (see Proposition \ref{natext}). The geometric intuition behind this definition will be descriged in the next \S\ref{crosssec} where we connect it to the geodesic flow.  If $u$ and $v$ denote inverse slope coordinates and $(u,v) \in \Sigma \times \Sigma_0$, let us set:
\begin{equation}\label{natextdef}
 \hat{F}( u,v )  :=  ( \gamma \nu_i [u ], \gamma \nu_i [v] ) \quad \mathrm{if} \,  u \in \Sigma_i, \quad i=1, \dots, 7.
\end{equation}
We remark  here that the action by fractional linear transformations is the same on both coordinates, but it is  only determined by the sector to which $u$ belongs. Thus, the map consist of seven branches $\hat{F}_i$ each defined on $\Sigma_i \times \Sigma_0$, $i=1, \dots, 7$. 
Let us remark that $\hat{F}$ is constructed so that if $\pi(u,v)=u$ is the projection on the first coordinate, we have 
\be\label{projectionF}
\pi \hat{F}(u,v)  = F ( \pi (u,v)) = F(u) .
\ee
 
\begin{lemma}\label{bijectionlemma}
The map $\hat{F}$ maps  $\Sigma \times \Sigma_0$ to itself and is a bijection.
\end{lemma}
\begin{proof}
It is easy to check that $\hat{F} $ is injective. 
We will prove surjectivity.  Let us remark that $ \gamma \nu_i \Sigma_i = \gamma \Sigma_0 = \Sigma$, i.e. the domain $\Sigma_i \times \Sigma_0$, of each branch $F_i$ stretches in the $u$ direction to fully cover the $u$ domain. For the $v$ coordinate, let us remark that $\nu_i \Sigma_0 = \Sigma_{j(i)}$ where $j(i) \in \{ 1, \dots, 7\}$ is such that $\nu_{j(i)}= \nu_i^{-1}$. Thus, since $\gamma \bar \Sigma = \bar \Sigma_0$, we have $\gamma \nu_i \bar \Sigma_0 = \gamma \bar \Sigma_{j(i)} \subset \bar \Sigma_0$. Moreover, since as $i$ ranges over the set $ \{ 1, \dots, 7\}$, $j(i)$ assumes all values in $\{ 1, \dots, 7\}$, $\cup_{i=1 }^{ 7} \gamma \nu_i \Sigma_0 = \cup_{i=1 }^{ 7} \gamma \Sigma_{j(i)} = \Sigma_0$. This concludes the proof that $\hat{F}$ is a bijection. 
\end{proof}
\noindent
Let us remark that $\nu_i^2 = 1$ if $i$ is odd, so that $\nu_i \Sigma_0 = \Sigma_i$ for $i=1,3,7$. Moreover  one can check that $\nu_2 \Sigma_0 = \Sigma_6$, $\nu_4 \Sigma_0 = \Sigma_4$ and  $\nu_6 \Sigma_0 = \Sigma_2$. Thus the function $j(i)$ defined in the proof above is $j(i)=i$ for all $i\neq 2,6$ and $j(2)=6$, $j(6)=2$.

In angle coordinates $(\theta, \phi) \in [\pi/8, \pi]\times [0, \pi/8]$, the action of $\hat{F}$ can be visualized as in Figure \ref{rectangles}: each rectangle in the left domain represents one of the domains $\Sigma_i \times \Sigma_0$, $i=1, \dots, 7$, while each rectangle in the right domain in Figure \ref{rectangles} represents one of the images of $F(\Sigma_i \times \Sigma_0)$, as indicated by the labels  $i=1, \dots, 7$. 

\begin{figure}
\centering
{\includegraphics[width=1\textwidth]{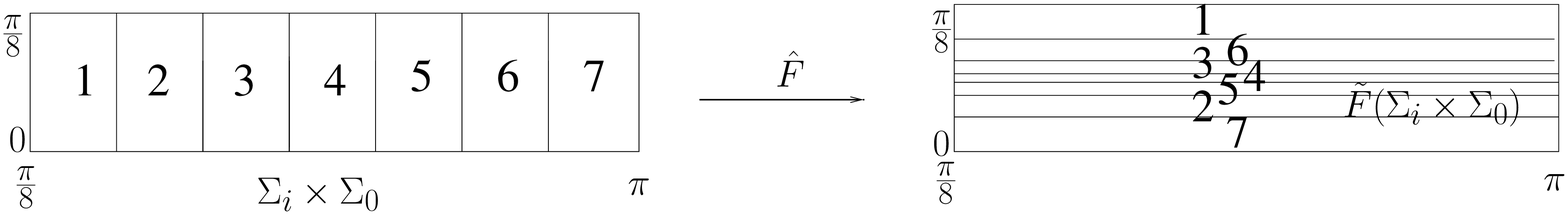}}
\caption{The action of the natural extension  $\hat{F}$ on $\Sigma \times \Sigma_0$  in angle coordinates.\label{rectangles}}
\end{figure}

\subsection{Backward octagon additive continued fraction expansion.}
From the definition of $\hat{F}$, as we already remarked, it is clear that $\hat{F}$ extends $F$ in the sense of (\ref{projectionF}).  
In order to prove that $\hat{F}$ gives a geometric realization of the natural extension  of $F|_\Sigma$, we define here explicitly the continued fraction symbolic coding of $\hat{F}$ that show that $\hat{F}$ is conjugate to a two-sided shift on seven symbols and thus is a natural extension. 
The following map $F^-$ (which is simply given by the action of the inverse of $\hat{F}$ on $v$-coordinates) defines what  might be called a \emph{backward} or \emph{dual} continued fraction expansion of $v \in \Sigma_0$.  
Let us remark that the sectors $\Sigma_i':= \gamma \nu_i\Sigma_0$ with $1\leq i \leq 7$ give a partition of $\Sigma_0$. Let us define
$ F^-(v)  =(\nu_i^{-1}  \gamma ) [v ] $ if  $ v \in  \Sigma_i'$ or, more  explicitly, computing $j(i)$ such that $\nu_{j(i)}= \nu_i^{-1}$ (see the remark after the proof of Lemma \ref{bijectionlemma}) and writing the intervals $\Sigma_i'$ in their increasing order inside $\Sigma_0$, we have:
\bes 
 F^-(v)  =  \begin{cases} (\nu_7  \gamma ) [v ] &  \text{if} \ v \in \Sigma_7'= \gamma \Sigma_7, \\
    (\nu_6  \gamma ) [v ] &  \text{if} \ v \in  \Sigma_2'=  \gamma\Sigma_6, \\ 
    (\nu_5  \gamma ) [v ] &   \text{if} \ v \in \Sigma_5'=  \gamma\Sigma_5, \\ 
    (\nu_4  \gamma  ) [v ] &  \text{if} \ v \in  \Sigma_4'= \gamma \Sigma_4, \\ 

      (\nu_3  \gamma  ) [v ] &   \text{if} \ v \in  \Sigma_3'= \gamma \Sigma_3,  \\ 
     (\nu_2  \gamma ) [v ] &  \text{if} \ v \in  \Sigma_6'= \gamma \Sigma_2, \\
    (\nu_1  \gamma  ) [v ] &   \text{if} \ v \in  \Sigma_1'= \gamma \Sigma_1 .
    \end{cases}
\ees
\noindent  The graph of $F^-$ in angle coordinates is shown in Figure \ref{backwardF}.
\begin{figure}
\centering
{\includegraphics[width=.4\textwidth]{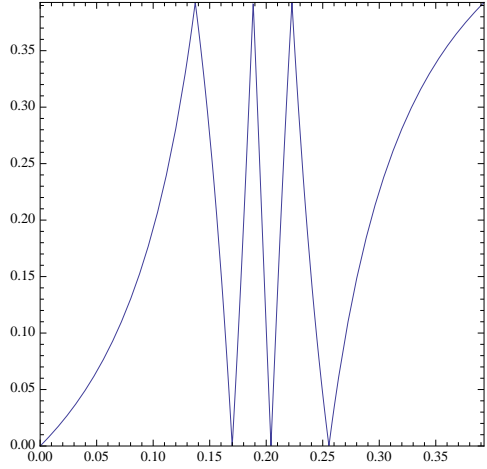}}
\caption{The  map $F^-$  in angle coordinates.\label{backwardF}}
\end{figure}
 The map $F^-$  is a continuous piecewise expanding map of $\Sigma_0$, which consists of seven branches which are given by piecewise linear transformations in the inverse of the  slope coordinates.  Thus, we can use it to define an expansion (the \emph{backward octagon additive continued fraction expansion}) as we did with the octagon Farey map $F$. We write $[v]_O^{-} = [s_{-1}, s_{-2}, \dots , s_{-k}, \dots ]$
and say that $s_{-1}$, $s_{-2}$, $\dots , s_{-k}, \dots$ are the entries of the backward additive continued fraction expansion of $v$ iff the itinerary of $v$ under  $F^-$  with respect to the partition $\{ \Sigma_i ', \, i=1, \dots, 7 \}$ is given by the sequence $\{ s_{-k}\}_{k\geq 1}$, i.e.~iff  we have $(F^-)^{k-1} v \in \Sigma_{s_{-k}}'$ for each $k\geq 1$. Equivalently, 
\bes 
[v]_O^{-} = [s_{-1}, s_{-2}, \dots , s_{-k}, \dots ]\quad \mathrm{iff} \quad v \in \cap_{k\in \mathbb{N}+} (F^-)^{-1}_{s_{-1}} (F^-)^{-1}_{s_{-2}} \cdots (F^-)^{-1}_{s_{-k}} \Sigma_0.
\ees
Given a point $ (u,v ) \in \Sigma \times \Sigma_0$,  combining the octagon additive continued fraction expansion of $u$ and the backward additive continued fraction expansion of $v$ we get the symbolic coding $\mathscr{C} :  \Sigma \times \Sigma_0 \rightarrow \{1,\dots, 7\}^{\mathbb{Z}} $
 given by 
\bes \mathscr{C} (u,v) = \{ s_i\}_{i\in \mathbb{Z}}, \quad \mathrm{where} \  u=[s_0, s_1, \dots ]_O , \quad v=[s_{-1}, s_{-2}, \dots ]^-_O . 
\ees
The following proposition shows that $\hat{F}$ is a geometric realization of the natural extension of the restriction of $F$ to its invariant set $\Sigma$. Let $\sigma$ be the full shift on $\{1,\dots, 7\}^\mathbb{Z}$, given by $\sigma ( \left\{s_i \right\}_{i\in \mathbb{Z}}) =  \left\{s_i \right\}_{i\in \mathbb{Z}}$ where $s_{i}'= s_{i+1} \in \{ 1, \dots, 7\}$. 
\begin{prop}\label{natext}
The map $\hat{F}$ is conjugate to the full shift  $\sigma$ on $\{1,\dots, 7\}^\mathbb{Z}$ by the symbolic coding, i.e.
\bes
\mathscr{C} (\hat{F} (u,v) ) = \sigma (\mathscr{C} (u,v )).
\ees
\end{prop}
\begin{proof}
If $\mathscr{C}(u,v) =  \left\{s_i \right\}_{i\in \mathbb{Z}}$, we know by definition of the forward and backward expansion, that for each $k$ there exits $u_k \in \Sigma$ and $v_k \in \Sigma_0$ such that
\begin{eqnarray}
 u &=& F^{-1}_{s_0} F^{-1}_{s_1}\cdots F^{-1}_{s_k} (u_k) =  \nu_{s_0}^{-1} \gamma \dots  \nu_{s_k}^{-1} \gamma [u_k]  \nonumber \\ 
 v &=& \nonumber (F^-)^{-1}_{s_{-1}} (F^-)^{-1}_{s_{-2}}\cdots (F^-)^{-1}_{s_{-k}} (v_k) = \gamma \nu_{s_{-1}} \dots \gamma   \nu_{s_{-k}}[ v_k ] \nonumber\end{eqnarray}
where the second equalities simply used the explicit definitions of the branches $F$ and $F^-$ and $\gamma^{-1}=\gamma$. Acting by $\hat{F}$, since the entry $s_0$ tells us that  $(u,v) \in \bar \Sigma_{s_0} \times \bar \Sigma_0 $, we have $(u',v'): = F(u,v) = (\gamma \nu_{s_0} [u], \gamma \nu_{s_0}[ v] )$. Thus
\begin{eqnarray}
 u'= \gamma \nu_{s_0} [u ] &=&  ( \gamma \nu_{s_0} \nu_{s_0}^{-1} \gamma ) \nu_{s_1}^{-1} \gamma  \dots  \nu_{s_k}^{-1} \gamma [u_k]  =\nu_{s_1}^{-1} \gamma  \dots  \nu_{s_k}^{-1} \gamma [u_k] ; \nonumber \\ 
 v'=\gamma \nu_{s_0} [v ] &=& \gamma \nu_{s_0}  \gamma \nu_{s_{-1}} \dots \gamma   \nu_{s_{-k}} [v_k ] .\nonumber
\end{eqnarray}
This shows as desired that $u'=[s_1, \dots, s_k, \dots ]_O$ and $v'= [s_0, s_{-1}, \dots, s_{-k}, \dots]^-_O$.
 
\end{proof}

\subsection{Invariant measure for the natural extension}
Let us use inverse slope coordinates $(u,v)$. Let $\hat{\mu}$ be the measure on $\Sigma \times \Sigma_0 $ whose density is given by
\bes
\frac{\ud u \ud v}{ (u-v)^2}.
\ees
\begin{lemma}\label{invariance}
The measure $\hat{\mu}$ is invariant under $\hat F$, i.e.~for each measurable set $D \subset \Sigma\times \Sigma_0$ we have $\hat{\mu} (F^{-1}(D)) = \hat{\mu} (D)$. 
\end{lemma}
\noindent The lemma uses the following simple identity.
\begin{lemma}\label{identitylemma}
Let $L$ be $L(x)= \frac{a x +b}{cx+d}$, with $ad-bc=\pm1$. Then for each $u,v \in \mathbb{R}$, 
\be \label{identity}
L'(u) L'(v) = \frac{\left(L(u)-L(v)\right)^2 }{(u-v)^2}.
\ee
\end{lemma}
\begin{proof}
The proof is a simple computation. From the expression of $L(x)$, we  have $L'(x) = \pm (cx+d)^{-2}$. Computing
\bes \begin{split}
L(u)-L(v) & = \frac{ (cv+d)(au+b) - (cu+d)(av+b)}{(cu+d)(cv+d)} \\ & = \frac{acuv +cvb +adu +db - acuv - bcu - adv - db   }{(cu+d)(cv+d)} = \frac{ (ad-bc) (u-v) }{(cu+d)(cv+d) }
\end{split}
\ees
and using that $ad-bc=\pm1$ and  $L'(u)L'(v) = (cu+d)^{-2}(av+b)^{-2}$, we get (\ref{identity}).
\end{proof} 

\begin{proofof}{Lemma}{invariance}
We have to check that for each measurable set $D \subset \Sigma\times \Sigma_0$ we have $\hat{\mu} (\hat{F}^{-1}(D)) = \hat{\mu} (D)$. Let us first decompose $D$ as a disjoint union $\cup_{i=1, \dots, 7} D_i$ where $D_i = D \cap F(\Sigma_i \times \Sigma_0 )$.  Thus, if we prove that for each $i$ we have $\hat{\mu} (\hat{F}^{-1}D_i) = D_i$,  we have $\hat{\mu} \left(\hat{F}^{-1} (D) \right) = \hat{\mu} \left(\hat{F}^{-1}( \cup_i D_i) \right) = \sum_i \hat{\mu} \left(\hat{F}^{-1} (D_i) \right) = \sum_i \hat{\mu} \left(D_i\right) = \hat{\mu} (D)$.

Let us remark that if $(x,y) \in  \hat{F}^{-1}(D_i) $, since $\hat{F}^{-1}(D_i) \subset \Sigma_i \times \Sigma_0$, $\hat F(x,y) = (\gamma \nu_i[ x], \gamma \nu_i [y])$.  Thus, using the change of variables $u = \gamma \nu_i [x]$ and $v = \gamma \nu_i [y]$ and applying  Lemma \ref{identity} to $L (z) := \nu_i^{-1} \gamma [z]$, we have
\bes
\begin{split}
\hat{\mu} \left( \hat{F}^{-1}(D_i)\right) & = 
\int \!\! \int_{\hat{F}^{-1}(D_i)} \frac{1}{(x-y)^{2}}  \ \ud x  \ \ud y  = \int \!\! \int_{D_i} \frac{   L'(u) L'(v)  }{(  \nu_i^{-1}\gamma [ u] -\nu_i^{-1}\gamma[ v ] )^{2}}  \ \ud u  \ \ud v  = \\ & =  \int \!\! \int_{D_i} \frac{1}{(u-v)^2} \ \ud u  \ \ud v = {\hat{\mu}}(D_i)  . 
\end{split}
\ees

\subsection{Invariant measure for the Octagon Farey map} 
Integrating the invariant measure $\hat{\mu}$ for $\hat{F}$ along the $v$-fibers we get the following (infinite) invariant measure.
\begin{prop}\label{invmeasureF}
The measure $\mu$ on $\Sigma$ whose density in the $u$ coordinates is given by
\bes
\frac{\ud u }{ (1+\sqrt{2} - u ) }
\ees 
is invariant under the octagon Farey map $F$. 
\end{prop}
\begin{proof}
Since $\pi \hat{F} = F \pi$ and ${\hat{\mu}}$ is invariant under $\hat{F}$, the pull back measure $\mu : = \pi_{*} \hat{ \mu}$ given by $\mu (D) = \hat \mu (\pi^{-1}( D)) $ for each measurable $D \subset \Sigma$ is invariant under $F$. To compute the density of this measure, it is enough to integrate the density of $\hat{\mu}$ along $v$-fibers. Recalling that the domain $\Sigma_0$ of the $v$ coordinate is $[ \cot(\pi/8), \infty]$  and that $\cot{\pi/8}= 1+ \sqrt{2}$ and using the change of variables $t= u-v$, we get the density
\bes
\int_{1+ \sqrt{2}}^{+ \infty} \frac{\ud v }{(u-v)^2} = \int_{u- 1- \sqrt{2}}^{-\infty}  \frac{- \ud t }{t^2} = \left. \frac{1}{t} \right|_{u-1-\sqrt{2}}^{-\infty} = \frac{1}{1+\sqrt{2}-u}.
\ees
\end{proof}

\begin{rem}
Since the domain $\Sigma $ of $u$ is $[-\infty, 1 + \sqrt{2}]$, the density blows up both at $-\infty$ and as $u\rightarrow (1+ \sqrt{2})^-$. Since the  singularities  are of type $1/u$, the density is not integrable. This shows that the invariant measure $\mu$ is infinite. 
\end{rem}

\begin{rem}
The invariant density can also be expressed in angle coordinates $\theta \in [\pi/8, \pi]$ and since $u= \cot \theta$, it is given by
\bes
\frac{\ud \theta}{\sin \theta \cos \theta - (1+ \sqrt{2})\sin^2 \theta } .
\ees
\end{rem}

\end{proofof}

\subsection{A cross section for the Teichm\"uller geodesic flow}\label{crosssec}

The natural extension of the Farey map $\hat{F}$ defined in the previous section has the following interpretation as a cross section of the Teichm\"uller geodesic flow on the Teichm\"uller orbifold of the octagon. 
Let us recall from \S\ref{Teichdisksec} that Teichm\"uller geodesics in ${\cal{\tilde M}}_A(S)$ are identified with hyperbolic geodesics in $T_1\mathbb{D}$. 
%
%
Let us consider the following section $\tilde{\mathscr{S}}$ of $T_1 \mathbb{D}$. Let us denote by $\mathscr{E}_i \subset \partial \Disk$ the smaller closed arc which lies under the side $E_i$. The section $\tilde{\mathscr{S}}$ is given by points $(z,w) \in T_1\mathbb{D} $ such that $z$ belongs to the side  $E_0$ of the ideal tessellation into octagons and $v$ is such that the geodesics $\hat g$ through $z$ in direction $v$ has backward endpoint belonging to the arc $\mathscr{E}_0$  of $\partial\mathbb{D}$  and  
  forward endpoint  belonging to the union of the arcs $\mathscr{E} := \mathscr{E}_1 \cup \dots \cup \mathscr{E}_7$ of $\partial\mathbb{D}$, i.~e.~forward end point in arc complementary to $\mathscr{E}_0$. 
 One can parametrize the section $\tilde{\mathscr{S}}$ using as coordinates $(g_+,g_-) \in \partial \Disk \times \partial \Disk$, the 
   coordinates of the forward and backward points  of the geodesics $\hat g$ through $(z,w) \in \tilde{\mathscr{S}}$.   
It is obvious from the definitions  that $\tilde{\mathscr{S}}= \{(g_+,g_-) \st (g_+,g_- ) \in \mathscr{E} \times \mathscr{E}_0 \}$.  
\begin{rem}\label{nextside}
From the construction, we see that the next side of $\mathscr{O}$ which is hit by the geodesic $\hat g$ starting at $(z,w)\in \mathscr{S}$ is the side $E_k$, where $1\leq k \leq 7$, if and only if the coordinates $(g_+,g_-)$ of $(z,w) $  are such that $g_+ \in \mathscr{E}_k$. 
\end{rem}
The section $\tilde{\mathscr{S}}$ projects to a section $\mathscr{S}$ on $T_1 \mathbb{D}/V_P({S})$, which is identified with a bundle ${\cal{M}}_{A}(S)$ over the Teichm\"uller orbifold ${\cal{M}}_{I}(S)$ associated to the octagon (see \S\ref{Teichorbsec}). Let us remark that a  geodesic $\hat g $ projected  to  $T_1 \mathbb{D}/V_P({S})$ gives a geodesic path $g$ which is a hyperbolic billiard trajectory in the fundamental triangle $\mathscr{F}$ (see \S\ref{Teichorbsec}). We have the following:
\begin{prop}\label{firstreturn}
The first return map of the geodesic flow on $T_1 \mathbb{D}/V_P(S)$ (which is  the hyperbolic billiard flow in $\mathscr{F}$) to the cross section $\mathscr{S}$ 
 is conjugated to the map $\hat{F}$. 
\end{prop}
\begin{proof}
Let $g$ be the geodesic path on $T_1 \mathbb{D}/V_P({S})$  starting at  $(z,w) \in \mathscr{S}$.  Let  $\hat{g}$ be the geodesic lift of $g$ to $T_1 \mathbb{D}$. Let $(g_+,g_-) \in \partial \Disk \times \partial \Disk$ be the endpoints of $\hat g$. 
Let us identify $\partial \Disk$ with $\mathbb{R}= \partial \mathbb{H}$ by extending  the identification $\phi: \mathbb{H} \rightarrow \mathbb{D}$  (given  in \S\ref{discsec}) and let us denote by $(x,y) \in \mathbb{R}^2 $ the point corresponding to $(g_+,g_-)$.
Since all the sides of $\mathscr{O}$ project to the same side on $T_1 \mathbb{D}/V_P(S)$,   the first return of $g$ to $\mathscr{S}$ is the projection on $T_1 \mathbb{D}/V_P(S)$ of the first point of $\hat{g}$ which hits one of the sides of $\mathscr{O}$. Let us assume that $g_+ \in \mathscr{E}_k$, for some $1\leq k\leq 7$. By Remark  \ref{uthetacorrespondence}, this assumption is equivalent to assuming $ -1/x = u \in \Sigma_k$, where $u=\cot \theta$ is the inverse of the slope coordinate on $\mathbb{RP}^1$. Moreover, by Remark \ref{nextside}, $\hat{g}$ leaves $\mathscr{O}$ by crossing the side $E_i$.  The projection to $T_1 \mathbb{D}/V_P(S)$ can be thus obtained applying the right action of the element $\nu_i^{-1} \gamma $, which maps $E_i$ to $E_0$ and $\mathscr{O}_i$ to $\mathscr{O}$. This maps  $(x,y)$ to  $\left( \frac{a x + c}{bx + d}, \frac{a y + c}{by + d}\right)$, where $ \left( \begin{smallmatrix} a_i & b_i \\ c_i & d_i  \end{smallmatrix} \right) = \nu_{i}^{-1} \gamma $.  If we hence apply the conjugacy  $(x(u), y(v)) = ( -1/u, -1/v)$, reasoning as in the proof of Lemma \ref{limitpoint}, this action is conjugated  to $(u,v)\mapsto (\gamma\nu_i[x],\gamma \nu_i[y])$, which is   exactly the action  of the map $\hat{F}$ for $u\in \Sigma_i$ (see (\ref{natextdef})). 
\end{proof}

Let us now show that the labeling of the edges of the tree $\mathscr{T}$ described in \S\ref{renormcutseq} is naturally related to a symbolic coding of the geodesic flow on the Teichm\"uller orbifold. 
Given a geodesic ray $r_{\theta}= \{ g_t^\theta \cdot S_O\}_{t\geq 0}$, let $p_{\theta}$ the path on the tree $\mathscr{T}$ shadowing $r_{\theta}$ and let $s_0, s_1, \dots$ be the labels of the edges of $p_{\theta}$. The first label $s_0$ simply tells us that $r_{\theta}$  will first hit the side $E_{s_0} $ of $\mathscr{O}$. Let $t_{\theta}$ be the corresponding hitting time. Consider the rest of the ray $\{g_t^{\theta} \cdot S_O\}_{t\geq t_{\theta}}$  and call $g_{\theta}$ its projection on $\mathbb{D}/V_P(S_O)$, which is a  path on the Teichm\"uller orbifold, starting on the cross section. Consider the first return map of $g_{\theta}$ to  the section. The section can be partitioned into sets $\mathscr{S}_i$, $1\leq i \leq k$, where $\mathscr{S}_i$ is  the projection to $\mathscr{S} $ of points $(z,w)$ of  $\tilde{\mathscr{S}}$ so that the geodesic starting at $(z,w)$ first leaves the octagon $\mathscr{O}$ through the side $E_i$ of $\mathscr{O}$. We can code the first return map by assigning the symbol $i$ to each set $\mathscr{S}_i$, so that the $k^{th}$ symbol of the coding is $i$ if and only if the $k^{th}$ return belongs to the set $\mathscr{S}_i$.  Then, combining Proposition \ref{firstreturn} with (\ref{projectionF}) and Proposition \ref{CFandcuttseq}, we have the following. 
\begin{cor} 
This symbolic coding of the first return map of $g_{\theta}$ to the cross-section is given by the labeling $s_1, s_2, \dots $ of the edges of the path $p_{\theta}$.
\end{cor}

\section{The octagon Gauss map}\label{gausssec}
Just as in the case for the classical Farey map, the invariant measure for the octagon Farey map $F$ is not finite. 
We will define an induced map $G $ from the octagon Farey map $F$ which admits a \emph{finite} invariant measure. The map $G $ is constructed from $F$ in a manner analogous to the way in which the classical Gauss map is constructed from the Farey map. Thus, we will call it the \emph{octagon Gauss map}.
Let us remark that $F$ has two parabolic fixed points, $\pi/8$ and $\pi$, which belong to the branches $F_1$ and $F_7 $ respectively. Let us define $G$ by putting together all consecutive iterates of the branch $F_1$ and of the branch $F_7$ respectively, as follows. 

If $u \in \Sigma_1$, let us define $n_1(u) = \min \{ n \in \mathbb{N}$ such that $F^n(u) \notin \Sigma_1 \}$. Similarly, if $u \in \Sigma_7$, let us define $n_7(u) = \min \{ n \in \mathbb{N}$ such that $F^n(u) \notin \Sigma_7 \}$. Then, recalling the definition of $F$ in \S\ref{Fareymapsec}, we can write
\be\label{Gaussdef1}
G(u) = \left\{ 
\begin{array}{ll}
F^{n_1(u)}(u) & \mathrm{if} \ u \in \bar\Sigma_1 \\
F(u) & \mathrm{if} \ u \in \bar\Sigma_2 \cup \bar\Sigma_3 \cup \bar\Sigma_4 \cup \bar\Sigma_5 \cup \bar\Sigma_6 \\
F^{n_7(u)}(u) & \mathrm{if} \ u \in \bar\Sigma_7 
\end{array}
\right.
\ee
More explicitly, one can check that if $u\in [F_1^{-1}(\pi/4),\pi/4 ]$, $n_1(u)=1$ and $G(u)= F(u)$ and by induction that if $u \in  [F_1^{-n}(\pi/4),F_1^{-n+1}(\pi/4) ]$, then $n_1(u)=n$. Let us denote by $\Sigma_{1,n}$ the sectors of $\mathbb{R}\mathbb{P}_1$ corresponding to  $u \in  [F_7^{-n+1}(7\pi/8),F_1^{-n}(7\pi/8) ]$. The sector $ \Sigma_{7,n}$ is exactly the one on which $n_7(u)=n$. Then we can write
\be\label{Gaussdef2}
G(u) = \left\{ 
\begin{array}{ll}
(\gamma \nu_1)^{n}[u] & \mathrm{if} \ u \in \Sigma_{1,n}, \quad n \in \mathbb{N}_* ; \\
(\gamma \nu_i) [u] & \mathrm{if} \ u \in \Sigma_i, \quad i=2, \dots, 6  ; \\
(\gamma \nu_7)^{n}[u] & \mathrm{if} \ u \in \Sigma_{7,n}, \quad n \in \mathbb{N}_* .
\end{array}
\right.
\ee

  In angle coordinates, the map $G$ is defined on $[\pi/8, \pi]$. 
  The graph of $G$ expressed in angle coordinates is shown in Figure \ref{Gaussgraph}.
  \begin{figure}
\centering
{\includegraphics[width=.5\textwidth]{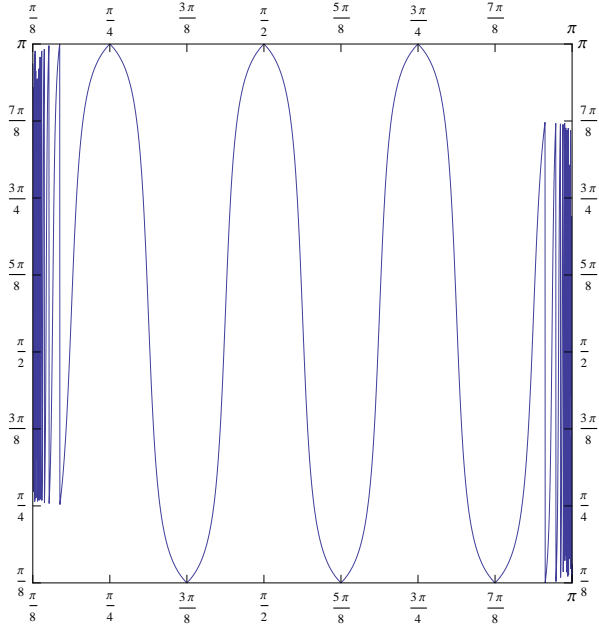}}
\caption{The octagon Gauss map $G$  in angle coordinates.\label{Gaussgraph}}
\end{figure}
The map has has countably many branches. We will denote by $G_{i}$, $i=2,\dots, 6$ and $G_{k,n}$, $k\in \{1,7\}$, $n\in \mathbb{N}_*$ the branches obtained restricting $G$ to $\Sigma_i$ and $\Sigma_{k,n}$ respectively. 
 \begin{rem} \label{images}
  Let us remark that while the image $G(\Sigma_i) = [\pi/8, \pi] $ for all $i=2, \dots, 6$, for each $n\in \mathbb{N}_*$ we have $G(\Sigma_{1,n}) = [2\pi/8, \pi] $ and  for each $n\in \mathbb{N}_*$ we have $G(\Sigma_{7,n}) = [\pi/8, 7\pi/8] $.
\end{rem} 

\paragraph{Natural extension and invariant measure for the octagon Gauss map.}
The natural extension $\hat{G}$ of $G$ is defined analogously to the natural extension $\hat{F}$ of $F$.  If $u$ and $v$ denote inverse slope coordinates and $(u,v) \in \Sigma \times \Sigma_0$, let us set:
\begin{equation}\label{natextGaussdef}
 \hat{G}( u,v )  :=  
\left\{ 
\begin{array}{ll}
\left((\gamma \nu_1)^{n}[u], (\gamma \nu_1)^{n}[u]\right) & \mathrm{if} \ u \in \Sigma_{1,n}, \quad n \in \mathbb{N}_* ; \\
\left((\gamma \nu_i) [u], (\gamma \nu_i) [v]\right) & \mathrm{if} \ u \in \Sigma_i, \quad i=2, \dots, 6  ; \\
\left( ( \gamma \nu_7)^{n}[u] , (\gamma \nu_7)^{n}[v] \right) & \mathrm{if} \ u \in \Sigma_{7,n}, \quad n \in \mathbb{N}_* .
\end{array}
\right. 
\end{equation}
Again  here the action is the same on both coordinates and is  only determined by the sector to which $u$ belongs.
The projection of $\hat{G}$ on the first coordinate is clearly $G$.

Let us define the following domain $D_{\hat{G}}$ of $\mathbb{R}\mathbb{P}^1$ giving its angle coordinates:
\be\label{decomp1}
D_{\hat{G}}: = D_1 \cup D_* \cup D_7, \quad \mathrm{where}\  \left\{ \begin{array}{l}
D_1 = \Sigma_1 \times [0, \gamma ( 2\pi/8) ] \\
D_* = [{2\pi}/{8}, {7\pi}/{8}] \times \Sigma_0 \\
D_7 = \Sigma_7 \times \left[  \gamma (7\pi/8) ,\pi/8  \right]
\end{array}
\right.
\ee
Figure \ref{domainGhat} shows two copies of domain $D_{\hat{G}}$ in angle coordinates: in the left copy, the sets $D_1, D_*$ and $D_7$ are shown. The right copy of the domain illustrates the following alternative decomposition, that can be easily checked:
\be\label{decomp2}
D_{\hat{G}}: = D_1' \cup D_*' \cup D_7', \quad \mathrm{where} \, \left\{ \begin{array}{l}
D_1' = [2\pi/8 , \pi ]  \times [ \gamma ( 2\pi/8), \pi/8 ] \\
D_*' = \Sigma  \times [ \gamma ( 7\pi/8), \gamma (2\pi/8) ]  \\
D_7' =  [ 0,  7\pi/8 ]  \times \left[ 0, \gamma (7\pi/8)  \right]
\end{array}
\right.
\ee

\begin{figure}
\centering
{\includegraphics[width=1\textwidth]{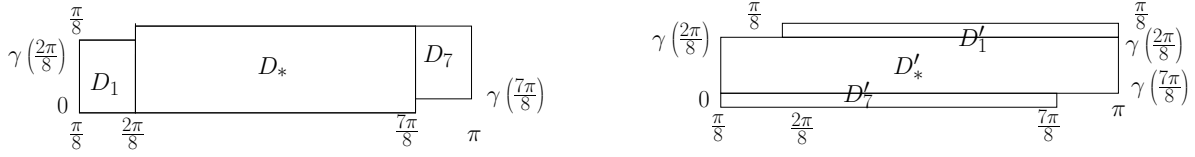}}
\caption{Domain and codomain of natural extension $\hat{G}$ of the octagon Gauss map $G$.\label{domainGhat}}
\end{figure}

\begin{lemma}
The map $\hat{G}$ maps  $D_{\hat{G}}$ to itself and is a bijection. Moreover, it preserves a measure whose density  is given by
$\frac{\ud u \ud v}{ (u-v)^2}$. 
\end{lemma}
\begin{proof}
As for $\hat{F}$, injectivity is immediate and we just have to check that $\hat{G}(D_{\hat{G}})=D_{\hat{G}}$. Since $D_* = \cup_{i=2}^6 \Sigma_i \times \Sigma_0$, one can easily see that $\hat{G}(D_*)= \hat{F} (D_*)= D_*'$. Thus, in view of the two decompositions (\ref{decomp1}, \ref{decomp2}), it is enough to prove that $\hat{G} (D_1)= D_1' $ and $\hat{G} (D_7)= D_7' $. 

For $n \in \mathbb{N}_*$, set $D_{1,n}:= \Sigma_{1,n}\times [0, \gamma(2\pi/8)]$, so that $D_1 = \cup_n D_{1,n}$. From the definition of $\hat{G}$, we see that if $(u,v) \in D_{1,n}$, then $\hat{G}(u,v) = ((\gamma \nu_1)^{n}[u], (\gamma \nu_1)^{n}[v])$.  From Remark \ref{images}, each $D_{1,n}$ stretches in the $u$-direction under the action of $\hat{G}$ to cover the full width of of $D_1'$, i.~e.~the $u$-interval $[2\pi/8, \pi ]$.  
Let us remark that $2\pi/8= \nu_1 (0)$ and that we have 
\bes
(\gamma \nu_1)^n \left[0, \gamma (2\pi/8) \right] = \left[(\gamma \nu_1)^n[0],(\gamma \nu_1)^{n+1} [0]  \right]. 
\ees 
Moreover, we have $\lim_{n\rightarrow \infty} (\gamma \nu_1)^n[0] = \pi/8$, since $\{ (\gamma \nu_1)^n[0] \}_{n\in \mathbb{N}}$ is a monotone sequence contained in $\Sigma_0$ and $\pi/8$ is the only fixed point of $\gamma \nu_1$ in $\Sigma_0$. Thus,
\bes
\bigcup_{n=1}^{\infty} (\gamma \nu_1)^n \left[0, \gamma (2\pi/8) \right] = \left[ \gamma (2\pi/8) , \pi/8  \right].
\ees
This show that $\bigcup_n \hat{G} (D_{1,n}) =D_1 '$. Analogously, defining $D_{7,n}:=  \Sigma_{1,n} \times [   \gamma (7\pi/8), \pi/8]$, one can show that $\bigcup_n \hat{G} (D_{7,n}) =D_7 '$.

The invariance of the measure given by the density $\frac{\ud u \ud v}{ (u-v)^2}$ is analogous to the proof of Lemma \ref{invariance}. It is enough to decompose any given measurable set by intersecting it with $\hat{G}({D_{k,n}})$, $k\in{1,7}$, $n\in \mathbb{N}_*$ and $\hat{G}(D_{*})$ and remark that, since on each $\hat{G}^{-1}$ in inverse of slope coordinates  is given by a fractional linear transformation, one can use Lemma \ref{identitylemma} to check invariance of the measure.  
\end{proof}

\begin{prop}
The measure $\mu_{G}$ on $\Sigma$ whose density in the $u$ inverse of the slope coordinate is

\be \label{densitymeasure} 
\begin{cases}\frac{\ud u }{u- \cot(\gamma(7\pi/8))} - \frac{\ud u }{1+ \sqrt{2}-u}, 
    & \text{   if } \ u \in [ -\infty, -1 - \sqrt{2} ] \\
   \frac{\ud u }{1+ \sqrt{2}- u} , 
    & \text{   if } \ u \in [ -1 - \sqrt{2}, 1 ]\\
  \frac{\ud u }{\sqrt{2}- u} , 
    & \text{   if } \ u \in [ 1, 1 + \sqrt{2} ] 
    \end{cases}
\ee

is invariant under the octagon Gauss map $G$. 
\end{prop}
\begin{proof}
As in the proof of Proposition \ref{invmeasureF}, it is enough to compute the density of the invariant measure for $G$ by integrating the invariant density for $\hat{G}$ along the fibers of the domain $D_{\hat{G}}$. This shows that the expressions in equation (\ref{densitymeasure}) are obtained by  evaluating the following integrals:
\begin{align*}
 & \int_{\cot(\gamma(2\pi/8))}^{+\infty}  \frac{\ud v}{ (u-v)^2} \quad  \text{  if  } \ u \in \Sigma_1, \\
 &\int_{\cot(\pi/8)}^{+\infty}  \frac{\ud v}{ (u-v)^2} \quad    \text{  if  } \ u \in \Sigma_2 \cup \Sigma_3 \cup \Sigma_4 \cup \Sigma_5 \cup \Sigma_6 ,\\
 & \int_{\cot(\pi/8)}^{\cot(\gamma(7\pi/8))  }  \frac{\ud v}{ (u-v)^2} \quad  \text{  if  } \ u \in \Sigma_7 ,
\end{align*}
and using the fact that $\cot(\pi/8)= 1+ \sqrt{2} = - \cot(7\pi/8)$ and $\cot(\gamma (2\pi/8))= \sqrt{2}$.  
\end{proof}

\section*{Acknowledgments}
We would like to acknowledge the hospitality and the support given from  the Max Planck Institute fur Mathematics during the special program on ``Dynamical Numbers". 

\bibliography{bibliooctagongeodesic}
\bibliographystyle{alpha}

{\small \rmfamily{  DEPARTMENT OF MATHEMATICS, CORNELL UNIVERSITY, ITHACA, NY, 14853, USA}}

{\it E-mail address: }\texttt{smillie@math.cornell.edu} 

\vspace{5mm}

{\small \rmfamily{SCHOOL OF MATHEMATICS, UNIVERSITY OF BRISTOL, BRISTOL, BS8 1TW, UK }}

{\it E-mail address: }\texttt{corinna.ulcigrai@bristol.ac.uk} 
 
\end{document}